\documentclass[11pt,letterpaper]{article}
\usepackage{fullpage}

\pdfoutput=1

\def\TEXPAD{0}

\usepackage{amsmath,amsthm,amssymb}
\usepackage{mathtools}
\usepackage{mathrsfs}
\usepackage[shortlabels]{enumitem}

\usepackage{comment}
\usepackage{hyperref}
\usepackage[nameinlink,capitalise,noabbrev]{cleveref}
\crefformat{equation}{(#2#1#3)}

\usepackage[usenames,dvipsnames]{xcolor}
\usepackage{graphicx}
\usepackage{float}
\usepackage{tikz}
\usetikzlibrary{decorations.pathreplacing,calc}
\colorlet{mygreen}{ForestGreen!60!LimeGreen}
\colorlet{graphyellow}{yellow}
\colorlet{graphblue}{blue!60!cyan}
\colorlet{graphred}{red}
\colorlet{graphgreen}{mygreen}
\colorlet{graphpurple}{violet!90!black}
\colorlet{commentblue}{blue}
\colorlet{commentpurple}{violet}

\usetikzlibrary{calc, arrows.meta, positioning, shapes.geometric, bending, decorations.markings}

\newcommand{\resettikzparameters}{
	\def\pinwheelopacity{0.4}
	\def\circletriangleangle{240}
}
\resettikzparameters

\tikzset{vtx/.style={circle, fill, inner sep=1.5pt}}
\tikzset{openvtx/.style={circle, draw, inner sep=1.5pt}}
\tikzset{
	circle arrow/.style={},
	purple opacity/.style={fill opacity=0.1},
	edge color/.style={},
}

\newcounter{uid}
\tikzset{redvtx/.style={vtx, draw, fill=graphred}}
\tikzset{bluevtx/.style={vtx, draw, fill=graphblue}}
\tikzset{greenvtx/.style={vtx, draw, fill=LimeGreen}}
\tikzset{yellowvtx/.style={vtx, draw, fill=graphyellow}}
\tikzset{whitevtx/.style={vtx, draw, fill=white}}

\tikzset{unlabeled triangle/.style={baseline=0.8ex, scale=0.6}}
\tikzset{labeled triangle/.style={baseline=1ex,scale=0.6,every node/.style={scale=0.7}}}

\newcommand{\tikzmaketriangle}[4][]
{\path (0,0) coordinate (#1#3) -- (1,0) coordinate (#1#4) -- (60:1) coordinate (#1#2);
}
\newcommand{\tikzlabeltriangle}[7][]
{\path (#1#2) node[above=-1pt] {$#5$}
		(#1#3) node[left=-1pt] {$#6$}
		(#1#4) node[right=-1pt] {$#7$};
}

\newcommand{\tikzmakepinwheel}[4][] {
	\path (#1#2) -- (#1#3) coordinate[pos=0.5] (#1#2#3)
			-- (#1#4) coordinate[pos=0.5] (#1#3#4)
			-- (#1#2) coordinate[pos=0.5] (#1#2#4)
			-- (#1#3#4) coordinate[pos=2/3] (#1#2#3#4);
	\if\pinwheelopacity0
	\else
	\fill[draw=none,fill=graphpurple!50!white,fill opacity=\pinwheelopacity]
		(#1#2) -- (#1#2#3#4) -- (#1#2#4) -- (#1#2)
		(#1#3) -- (#1#2#3#4) -- (#1#2#3) -- (#1#3)
		(#1#4) -- (#1#2#3#4) -- (#1#3#4) -- (#1#4);
	\fi
	\def\pinwheelopacity{0}
	\draw[graphpurple, fill=graphpurple, fill opacity=0.8]
		(#1#2) -- (#1#2#3#4) -- (#1#2#3) -- cycle
		(#1#3) -- (#1#2#3#4) -- (#1#3#4) -- cycle
		(#1#4) -- (#1#2#3#4) -- (#1#2#4) -- cycle;
	\draw[graphpurple, line width=0.45pt] (#1#2) -- (#1#3) -- (#1#4) -- cycle;
}

\newcommand{\tikzmakeshiftedvtxs}[5]{
	\stepcounter{uid}
	\coordinate (\theuid med) at ($1/3*(#1)+1/3*(#2)+1/3*(#3)$); 
	\coordinate (\theuid x) at ($(#1)!#5!(\theuid med)+#4$);
	\coordinate (\theuid y) at ($(#2)!#5!(\theuid med)+#4$);
	\coordinate (\theuid z) at ($(#3)!#5!(\theuid med)+#4$);
}
\newcommand{\trianglecommands}[4][]{
\expandafter\newcommand\csname #2\endcsname{
	\tikz[unlabeled triangle, #1] {
		\tikzmaketriangle #4;
		\csname tikzmake#3\endcsname xyz;
}}
\expandafter\newcommand\csname #2labeled\endcsname[3]{
	\tikz[labeled triangle, #1] {
		\tikzmaketriangle #4;
		\tikzlabeltriangle #4{##1}{##2}{##3}
		\csname tikzmake#3\endcsname xyz;
}}
\expandafter\newcommand\csname tikzshifted#2\endcsname[5]{
	\begin{scope}[#1]
	\tikzmakeshiftedvtxs{##1}{##2}{##3}{##4}{##5};
	\csname tikzmake#3\endcsname[\theuid]xyz;
	\end{scope}
}
}

\trianglecommands{plaintriangle}{plaintriangle}{xyz}
\trianglecommands{bluetriangle}{bluetriangle}{xyz}
\trianglecommands{graytriangle}{graytriangle}{xyz}
\trianglecommands{rgbtriangle}{rgbtriangle}{xyz}
\trianglecommands{brgtriangle}{rgbtriangle}{zxy}

\trianglecommands{pointedtriangle}{pointedtriangle}{xyz}
\trianglecommands{pointedtriangleleft}{pointedtriangle}{yxz}
\trianglecommands{pointedtriangleright}{pointedtriangle}{zyx}

\trianglecommands[edge color/.prefix style={graphred},vtx color/.prefix style={graphyellow}]
{rededgetriangle}{edgetriangle}{xyz}
\trianglecommands[edge color/.prefix style={graphyellow},vtx color/.prefix style={graphred}]
{yellowedgetriangle}{edgetriangle}{xyz}
\trianglecommands[edge color/.prefix style={graphred},vtx color/.prefix style={graphyellow}]
{rededgetriangleright}{edgetriangle}{yxz}
\trianglecommands[edge color/.prefix style={graphyellow},vtx color/.prefix style={graphred}]
{yellowedgetriangleright}{edgetriangle}{yxz}

\trianglecommands{pinwheel}{pinwheel}{xyz}
\trianglecommands{pinwheelrev}{pinwheel}{yxz}
\trianglecommands{circletriangle}{circletriangle}{xyz}
\trianglecommands{circletrianglerev}{circletriangle}{xzy}

\tikzset{unlabeled tetrahedron/.style={baseline=0.4ex,scale=0.6}}
\tikzset{labeled tetrahedron/.style={baseline=0.42ex,scale=0.6,every node/.style={scale=0.7}}}

\newcommand{\tikzmaketetrahedron}[5][]
{\path (0,0.05) coordinate (#1#3)
	(0.75,-0.2) coordinate (#1#4)
	(1,0.33) coordinate (#1#5)
	(0.5,0.85) coordinate (#1#2);
}
\newcommand{\tikzlabeltetrahedron}[9][]
{\path (#1#2) node[above=-1.3pt] {$#6$}
	(#1#3) node[left=-1pt] {$#7$}
	(#1#4) node[below right=-3pt] {$#8$}
	(#1#5) node[right=-1pt] {$#9$};
}

\newcommand{\tetrahedroncommands}[4][]{
\expandafter\newcommand\csname #2\endcsname{
	\tikz[unlabeled tetrahedron, #1] {
		\tikzmaketetrahedron wxyz;
		\csname tikzmake#3\endcsname #4;
}}
\expandafter\newcommand\csname #2labeled\endcsname[4]{
	\tikz[labeled tetrahedron, #1] {
		\tikzmaketetrahedron wxyz;
		\tikzlabeltetrahedron wxyz{##1}{##2}{##3}{##4};
		\csname tikzmake#3\endcsname #4;
}}
}
\tetrahedroncommands{pointedtetrahedron}{pointedtetrahedron}{wxyz}
\tetrahedroncommands{pointedtetrahedrontwisted}{pointedtetrahedron}{yxwz}
\tetrahedroncommands[edge1/.style={graphyellow}, edge2/.style=graphred]
	{edgetetrahedron}{edgetetrahedron}{wxyz}
\tetrahedroncommands[edge1/.style={graphred}, edge2/.style=graphyellow]
	{edgetetrahedronflipped}{edgetetrahedron}{wxyz}
\tetrahedroncommands{circletetrahedron}{circletetrahedron}{wxyz}

\tikzset{tetrahedron boundary/.style={
	inline/.style={baseline=0.48ex,scale=0.9,vtx/.append style={scale=0.6}, line width=0.4pt},
	big/.style={scale=2.7, line width=0.6pt,
		circle arrow/.append style={>={Latex[angle=50:3pt,scale=1.5]}}},
}}

\tikzset{hash1/.style={postaction={decorate}, 
        decoration={markings, mark=at position 0.5 with {\draw[-, line width=0.5mm] (0,-2pt) -- (0,2pt);}}}}
\tikzset{hash2/.style={postaction={decorate}, 
        decoration={markings,
        mark=at position 0.4 with {\draw[-, line width=0.5mm] (0,-2pt) -- (0,2pt);},
        mark=at position 0.6 with {\draw[-, line width=0.5mm] (0,-2pt) -- (0,2pt);}
        }}}
\tikzset{hash3/.style={postaction={decorate}, 
        decoration={markings,
        mark=at position 0.32 with {\draw[-, line width=0.47mm] (0,-2pt) -- (0,2pt);},
        mark=at position 0.5 with {\draw[-, line width=0.47mm] (0,-2pt) -- (0,2pt);},
        mark=at position 0.68 with {\draw[-, line width=0.47mm] (0,-2pt) -- (0,2pt);}
        }}}
\tikzset{arrow1/.style={>={
	Stealth[width=1.84mm,length=1.2mm]},
		postaction={decorate},
		decoration={markings, mark=at position 0.5 with
			{\draw[->] (-0.6mm,0) -- (0.6mm,0);}}}}
\tikzset{arrow2/.style={>={
	Stealth[width=1.84mm,length=1.2mm]},
		postaction={decorate},
		decoration={markings, mark=at position 0.5 with
			{\draw[->>] (-1.2mm,0) -- (1.2mm,0);}}}}
\tikzset{arrow3/.style={>={
	Stealth[width=1.84mm,length=1mm]},
		postaction={decorate},
		decoration={markings, mark=at position 0.5 with
			{\draw[->>>] (-1.5mm,0) -- (1.5mm,0);}}}}

\tikzset{tikz arrows/.style={
	line width=0.9pt,
	every node/.style={scale=0.8,inner sep=2pt},
	>={Stealth[width=1.84mm, length=2mm]},
	e0/.style={},e1/.style={},e2/.style={},e3/.style={},
red arrow/.style={
	graphred,
	e1/.style={arrow1},
	e2/.style={arrow2},
	e3/.style={arrow3},
	e0/.style={->}
},
blue edge/.style={
	graphblue,
	e1/.style={hash1},
	e2/.style={hash2},
	e3/.style={hash3},
	e0/.style={}
},
bipartite edge/.style={graphblue, >={Stealth[scale=1, red]}, ->}
}}

\newcommand{\arrowcommands}[3][]{
\expandafter\newcommand\csname#2\endcsname[1][#1]{
	\tikz[baseline=-0.6ex, tikz arrows]{
	\draw (0,0) coordinate[vtx] (1);
	\draw (0.75,0) coordinate[vtx] (2);
	\draw[#3,##1] (1) -- (2);
}}
\expandafter\newcommand\csname#2labeled\endcsname[3][#1]{
	\tikz[baseline=-0.6ex, tikz arrows]{
	\path (0,0) coordinate[vtx] (1) node[above] {$##2$};
	\draw (0.75,0) coordinate[vtx] (2) node[above] {$##3$};
	\draw[#3,##1] (1) -- (2);
}}
\expandafter\newcommand\csname#2xy\endcsname[1][#1]
{\csname#2labeled\endcsname[##1]{\vphantom{y}x}{y}}
}

\arrowcommands[e0]{redto}{red arrow}
\arrowcommands[e0]{blueedge}{blue edge}
\arrowcommands{greenedge}{graphgreen}
\arrowcommands{purpleto}{graphpurple, ->}
\arrowcommands{purpleot}{graphpurple, <-}
\arrowcommands{bipedge}{bipartite edge} 

\tikzset{unlabeled vtx triangle/.style={baseline=1ex, scale=0.7}} 
\tikzset{labeled vtx triangle/.style={baseline=1ex, scale=0.7}}

\newcommand{\tikzmakevtxtriangle}[4][]
{\path (0,0) coordinate[vtx] (#3) -- (1,0) coordinate[vtx] (#4) 
	-- (60:1) coordinate[vtx] (#2);
}
\newcommand{\tikzlabelvtxtriangle}[7][]
{\path (#1#2) node[left=0.5pt] {$#5$}
		(#1#3) node[left=0pt] {$#6$}
		(#1#4) node[right=0pt] {$#7$};
}

\newcommand{\triangleboundarycommands}[5][]{
\expandafter\newcommand\csname #2\endcsname[1][#1]{
	\tikz[unlabeled vtx triangle, tikz arrows] {
	\tikzmakevtxtriangle xyz;
	\draw[#3,##1] (y) -- (x);
	\draw[#4,##1] (z) -- (x);
	\draw[#5,##1] (y) -- (z);
}}
\expandafter\newcommand\csname #2labeled\endcsname[4][#1]{
	\tikz[labeled vtx triangle, tikz arrows, ##1] {
		\tikzmakevtxtriangle xyz;
		\tikzlabelvtxtriangle xyz{##2}{##3}{##4};
		\draw[#3,##1] (y) -- (x);
		\draw[#4,##1] (z) -- (x);
		\draw[#5,##1] (y) -- (z);
}}
}

\triangleboundarycommands[e0]{cherry}{red arrow}{red arrow}{blue edge}
\triangleboundarycommands[e0]{cherrypath}{red arrow}{draw=none,}{blue edge}
\triangleboundarycommands{greentri}{graphgreen}{graphgreen}{graphgreen}
\triangleboundarycommands{purpletri}{graphpurple, ->}{graphpurple, <-}{graphpurple, <-}
\triangleboundarycommands{purpletranstri}{graphpurple, ->}{graphpurple, ->}{graphpurple, ->}
\triangleboundarycommands{dirboundary}{bipartite edge}{bipartite edge}{blue edge}

\tikzset{graph/.style={edges/.style={draw=black!60, line width=0.6pt}}}

\newcommand{\tikzC}[1]{%
\ifnum #1=3
\tikz[baseline=-0.4ex,scale=0.18,graph,vtx/.append style={scale=0.65}] {
	\foreach \i in {0,1,2} {
		\coordinate[vtx] (v\i) at (120*\i + 90:1);
	}
	\draw (v0) -- (v1) -- (v2) -- (v0);
}%
\else\ifnum #1=4
\tikz[baseline=-0.7ex,scale=0.2,graph,vtx/.append style={scale=0.65}] {
	\foreach \i in {0,1,2,3} {
		\coordinate[vtx] (v\i) at (90*\i + 45:1);
	}
	\draw (v0) -- (v1) -- (v2) -- (v3) -- (v0);
}%
\else\ifnum #1=5
\tikz[baseline=-0.68ex,scale=0.22,graph,vtx/.append style={scale=0.65}] {
	\foreach \i in {0,1,2,3,4} {
		\coordinate[vtx] (v\i) at (72*\i + 90:1);
	}
	\draw (v0) -- (v1) -- (v2) -- (v3) -- (v4) -- (v0);
}%
\fi\fi\fi}

\newcommand{\tikzK}[1]{%
\ifnum #1=2
\tikz[baseline=-0.45ex,scale=0.18,graph,vtx/.append style={scale=0.65}] {
	\foreach \i in {0,1} {
		\coordinate[vtx] (v\i) at (\i*2,0);
	}
	\draw (v0) -- (v1);
}%
\else\ifnum #1=4
\tikz[baseline=-0.7ex,scale=0.2,graph,vtx/.append style={scale=0.65}] {
	\foreach \i in {0,1,2,3} {
		\coordinate[vtx] (v\i) at (90*\i + 45:1);
	}
	\draw (v0) -- (v1) -- (v2) -- (v3) -- (v0) (v0) -- (v2) (v1) -- (v3);
}
\fi\fi}

\usepackage[numbers]{natbib}

\DeclareTextCompositeCommand{\v}{OT1}{l}{l\nobreak\hspace{-.1em}'}

\theoremstyle{definition}
\newtheorem{definition}{Definition}
\newtheorem{example}[definition]{Example}
\newtheorem{construction}[definition]{Construction}
\newtheorem{remark}[definition]{Remark}

\theoremstyle{plain}

\newtheorem{theorem}[definition]{Theorem}
\newtheorem{lemma}[definition]{Lemma}

\newtheorem{claim}[definition]{Claim}
\newtheorem{subclaim}[definition]{Subclaim}

\if\TEXPAD0
	\crefname{claim}{Claim}{Claims}
	\crefname{lemma}{Lemma}{Lemmas}
\fi

\def \l {\ell}

\def \ex {\mathrm{ex}}
\def \sm {\setminus}
\def \ce {\coloneqq}

\renewcommand{\le}{\leqslant}
\renewcommand{\ge}{\geqslant}
\renewcommand{\leq}{\leqslant}
\renewcommand{\geq}{\geqslant}

\def \eps {\varepsilon}
\def \es {\varnothing}
\renewcommand \b[2] {\binom{#1}{#2}}

\def \mC {\mathcal{C}}

\def \mF{\mathcal{F}}

\def \mH{\mathcal{H}}

\def \K {\mathcal{K}}

\newcommand{\sg}{\sigma}

\newcommand{\del}{\delta}
\newcommand{\Ga}{\Gamma}
\newcommand{\ga}{\gamma}

\newcommand{\texthom}{\text{-}\mathrm{hom}}
\newcommand{\C}[1][\l]{\mathcal{C}^r_{#1}}
\newcommand{\Ch}[1][\l]{\mathcal{C}^r_{#1}\texthom}

\newcommand{\Ckr}{\mathcal{C}_{\equiv k}^r}
\newcommand{\CkLr}{\mathcal{C}_{\equiv k, < L}^r}
\newcommand{\CkLrh}{\mathcal{C}_{\equiv k, < L}^r\texthom}
\newcommand{\Ckrh}{\mathcal{C}_{\equiv k}^r\texthom}
\newcommand{\sit}{S_i \times S_{r-i}}

\newcommand{\pich}{\gamma (\Ckr\texthom)}

\newcommand{\drmh}[1][\mH]{\delta_{r-1}(#1)}
\newcommand{\vx}[1][x]{\vec{#1}}
\newcommand{\vy}{\vec y}
\newcommand{\EdH}[1][\mH]{\vec E (#1)}
\def \fG {\mathfrak{G}}
\newcommand{\cyc}{\mathrm{cyc}}

\newcommand{\hp}{h^+_}
\newcommand{\hm}{h^-_}
\newcommand{\hpm}{h^{\pm}_}
\newcommand{\Gatr}{\Ga_{\mathrm{tr}}}

\title{A Jump in the Codegree Tur\'an Densities of Long Tight Cycles}
\author{
J\'ozsef Balogh\thanks{Department of Mathematics, University of Illinois Urbana-Champaign, Urbana, IL, USA, and Extremal Combinatorics and Probability Group (ECOPRO), Institute for Basic Science (IBS), Daejeon, South Korea. Partially supported by NSF grants RTG DMS-1937241, FRG DMS-2152488, (UIUC  Campus Research Board Award RB26026), the Simons Fellowship, Simons Collaboration grant [SFI-MPS-TSM-00013107, JB], and the Institute for Basic Science (IBS-R029-C4).
Email: \texttt{jobal@illinois.edu}. 
}
\and
Haoran Luo\thanks{Department of Mathematics, Statistics and Computer Science, University of Illinois Chicago, Chicago, Illinois 60607, USA. Research was partially performed while Luo was at the University of Illinois Urbana-Champaign and supported in part by Dr. James J. Woeppel Fellowship. Email: \texttt{haoranl8@uic.edu}. }
\and
Maya Sankar\thanks{School of Mathematics, Institute for Advanced Study, Princeton, New Jersey 08540, USA. Research was partially performed while Sankar was at Stanford University and supported in part by NSF GRFP Grant DGE-1656518 and a Hertz fellowship. Email: \texttt{mayars@ias.edu}.}
}

\date{}

\begin{document}
\maketitle

\begin{abstract}
We study the codegree Tur\'an density of $\mathcal{C}_\ell^r$, the $r$-uniform hypergraph tight cycle of length $\ell$.
A result of Han, Lo, and Sanhueza-Matamala states that if $\ell$ is sufficiently large and $r/\gcd(r,\ell)$ is even, then the codegree Tur\'an density of $\mathcal{C}_\ell^r$ is $1/2$.
We prove that whenever the latter assumption is not satisfied, there is a significant drop in the codegree Tur\'an density. That is, if $\ell$ is sufficiently large and $r/\gcd(r,\ell)$ is odd, then the codegree Tur\'an density of $\mathcal{C}_\ell^r$ can be at most $1/3$.
Moreover, this bound is tight for infinitely many uniformities $r$ and all sufficiently large $\ell$ in the corresponding residue classes modulo $r$.
Our proof makes use of a group-theoretic connection between Tur\'an-type theorems for tight cycles and ``oriented colorings'' of the edge set of a hypergraph.
\end{abstract}

\section{Introduction} \label{sec::Int}
For a family $\mF$ of $r$-uniform hypergraphs (henceforth abbreviated as \emph{$r$-graphs}), the \emph{Tur\'an number} of $\mF$, denoted by $\ex(n,\mF)$, is the maximum number of edges in an $n$-vertex $r$-graph that contains no member of $\mF$ (i.e.,\ is $\mF$-\emph{free}). 
In many cases, it is out of reach to determine this number exactly, and instead, many results focus on determining the asymptotic density of the extremal constructions. This is encapsulated in the \emph{Tur\'an density} of $\mF$, which is defined to be $\pi(\mF) \ce \lim_{n \to \infty} \ex(n,\mF) / \b{n}{r}$; the limit is known to exist. When $\mF = \{\mH\}$, we simply write $\ex(n,\mH)$ and $\pi(\mH)$.

Determining the Tur\'an numbers and Tur\'an densities of $r$-graphs is a central problem in extremal combinatorics. For graphs (the case $r=2$), we have a reasonably good understanding of this problem, as the 
  fundamental theorem of Erd\H{o}s--Stone--Simonovits \cite{erdos1946structure, erdos1966limit} shows that $\pi(\mF)=1-\frac 1{\chi(\mF)-1}$, where $\chi(\mF)$ is the minimum chromatic number of a graph in $\mF$. The lower bound comes from a balanced complete $(\chi(\mF)-1)$-partite graph, which clearly contains no graph of chromatic number at least $\chi(\mF)$. 

Conversely, for hypergraphs (the case $r \ge 3$), understanding Tur\'an densities is a notoriously difficult problem, and very little is known. For example, denote by $\K_4^3$  the complete $3$-graph on $4$ vertices, and $\K_4^{3-}$ the hypergraph obtained by removing a hyperedge from  a $\K_4^3$.  Tur\'an famously  conjectured $\pi(\K_4^3)=5/9$, see~\cite{turan1961research} for proposing the general problem. Later, Frankl and F\"uredi~\cite{frankl1984exact} proved that $\pi(\K_4^{3-})\ge 2/7$, and it is conjectured to be an equality.
Despite considerable efforts and the seeming simplicity of these hypergraphs, both conjectures remain widely open to this day. 
For further discussion of results and techniques in this area, we refer the reader to the surveys by Keevash~\cite{keevash2011hypergraph} and Balogh, Clemen, and Lidick\'y~\cite{balogh2022hypergraph}.

In this paper, we study a variant called \emph{codegree Tur\'an density}, introduced by Mubayi and Zhao~\cite{mubayi2007codegree}. Unlike graphs, hypergraphs admit several different notions of degree. 
In an $r$-graph $\mH$, it is natural to define the degree $d(S)$ of a set $S$ of vertices, which is  the number of edges $e$ containing $S$. Similarly, for each $1\leq i\leq r-1$, we  define the minimum $i$-degree $\delta_i(\mH)$ to be the minimum
$d(S)$ over all   $S$ subsets of vertices of size $i$. In this paper, we are primarily interested in $\delta_{r-1}(\mH)$, which we usually call the \emph{minimum codegree} of a hypergraph $\mH$.

For a family $\mF$ of $r$-graphs, the \emph{codegree Tur\'an number} $\ex_{r-1} (n,\mF)$ is the maximum possible value of $\delta_{r-1}(\mH)$ over all $n$-vertex $\mF$-free $r$-graphs $\mH$. Similarly to the standard Tur\'an number, this quantity is also usually infeasible to  be determined exactly. Instead, we study the \emph{codegree Tur\'an density} $\gamma(\mF)=\lim_{n\to \infty} \ex_{r-1} (n,\mF)/n$, which is  well-defined, see Proposition~1.2 in~\cite{mubayi2007codegree}.

When  $\mF$ consists of a clique, symmetrization arguments imply that $\pi(\mF)=\gamma(\mF)$. For $r\geq 3$,  similarly, the degrees of an extremal hypergraph are differing by at most $O(n^{r-1})$ holds under the analogous assumption that each $\mH\in \mF$ \emph{covers pairs} (see \cite{keevash2011hypergraph}). Again, for such $\mF$, we can reinterpret $\pi(\mF)$ as the limiting density, as $n\to\infty$, of an $\mF$-free $r$-graph $G$ on $n$ vertices which maximizes $\delta_1(G)$. 
Viewed from this perspective, it is natural to study $\mF$-free $r$-graphs maximizing other notions of minimum degree, and the minimum codegree is a natural candidate. 

Determining the codegree Tur\'an density of a family of hypergraphs $\mF$ is a challenging problem, and it
 has a very different flavor from the corresponding ordinary Tur\'an problem. 
 A conjecture of  Czygrinow and Nagle~\cite{czygrinow2001codegree} is that $\gamma(\K_4^3) = 1/2$.
Using the flag algebra method of Razborov~\cite{razborov2007flag}, Falgas-Ravry, Pikhurko, Vaughan, and Volec~\cite{falgasravry2023codegree} proved that $\gamma(\K_4^{3-}) = 1/4$.
See Table 1 in~\cite{balogh2022hypergraph} for more results and conjectures.

There are several kinds of hypergraph cycles: tight cycles, loose cycles, Berge cycles, etc.,  with many different extremal behaviors. The focus of this paper is tight cycles.
Given a length $\l > r \ge 2$, the \emph{tight $r$-uniform cycle} $\C$ of length $\l$ is the $r$-graph with vertex set $\{0, 1, \ldots, \l-1\}$ and edge set $\{\{i, i+1, \ldots, i+r-1\}: 0 \le i < \l\}$, where the addition is modulo $\l$. Note that $\mC_4^3$ is exactly $\K_4^3$. 

The Tur\'an problem for tight cycles has drawn considerable attention.
When $\l$ is a multiple of $r$, we have that $\C$ is $r$-partite, and a classical result of Erd\H{o}s~\cite{erdos1964extremal} gives $\pi(\C) = \gamma(\C) = 0$. For the case $r=3$, it is conjectured that $\pi(\mC_5^3) = 2\sqrt{3} -3$, see~\cite{falgasravry2012turan, mubayi2011hypergraph}. A recent breakthrough by Kam\v{c}ev, Letzter, and Pokrovskiy~\cite{kamcev2024turan} gives $\pi(\mC^3_\l) = 2\sqrt{3} -3$ for sufficiently large $\l$ not divisible by $3$. Using the flag algebra method of Razborov~\cite{razborov2007flag}, Bodn\'ar, Le\'on, Liu, and Pikhurko~\cite{bodnar2025turan} proved that $\pi(\{\K_4^3, \mC_5^3\}) = \pi (\mC_\l^3) = 2\sqrt{3}-3$ for every $\l \ge 7$ not divisible by $3$. 
Regarding codegree Tur\'an density, Piga, Sanhueza-Matamala, and Schacht~\cite{piga2026codegree} proved $\gamma(\mC^3_\l) = 1/3$ for every $\l \in \{10,13,16\}$ and every $\l \ge 19$ not divisible by $3$, and Ma~\cite{ma2024codegree} proved $\gamma(\mC^3_{\ell}) = 1/3$ when $\l \in \{11,14,17\}$. For $r=4$, Sankar~\cite{sankar2024turan} proved $\pi (\mC_\l^4) = 1/2$ for sufficiently large $\l$ not divisible by $4$; her work also implies $\gamma(\mC_\l^4)=1/2$ as an easy corollary. For other values of $r$ and $\l$, Han, Lo, and Sanhueza-Matamala~\cite{han2021covering} proved $\gamma(\C) = 1/2$ if $\l \ge 2r^2$ and $r / \gcd(r, \l)$ is even. The extremal construction (a special case of \cref{construction} below) is the complete oddly bipartite hypergraph --- this is an $r$-graph $\mH$ with vertex set equipartitioned as $V(\mH)=V_1\sqcup V_2$, whose edges are those $r$-sets containing an odd number of vertices from $V_1$ and, implicitly, an odd number of vertices from $V_2$, as $r$ is even.

We remark that some other variations of Tur\'an problems have also been studied for tight cycles, see for example~\cite{balogh2022hypergraph, bucic2023uniform, balogh2026positive}.
\medskip

As we mentioned earlier,  Han, Lo, and Sanhueza-Matamala~\cite{han2021covering} proved $\gamma(\C) = 1/2$ if $\l \ge 2r^2$ and $r / \gcd(r, \l)$ is even.
In this paper, we make progress on the case when $r / \gcd(r,\l)$ is an odd integer. Our main theorem is the following.
\begin{theorem} \label{thm::main}
    For every uniformity $r \ge 2$ and residue class $k \not\equiv 0 \pmod r$ such that $r / \gcd(r,k)$ is odd, there is an integer $L$ such that for every $\l > L$ with $\l \equiv k \pmod r$, we have
    \[
        \gamma (\C) \le 1/3.
    \]
\end{theorem}

This bound is tight for infinitely many pairs $(r,k)$, as the following construction shows.

\begin{construction} [See Construction 10.1 in~\cite{han2021covering}] \label{construction}
    Fix a uniformity $r$ and an integer $p$. Partition the vertex set $V$ into $V_1,\ldots, V_p$  such that their sizes differ by at most one. An $r$-set of vertices $\{x_{1}, x_{2}, \ldots, x_{r}\}$, where $x_j \in V_{i_j}$,  is an edge if and only if $\sum_{j=1}^r i_j \equiv 1 \pmod p$.
\end{construction}

Clearly, the hypergraph $\mH$ in \cref{construction} has minimum codegree $(1/p - o(1)) n$. 
Additionally, we show (see \cref{lem::conCkrhfree}) that this hypergraph is $\C$-free for every $\ell$ with $\gcd(r,\ell)\mid r/p$. Thus, by \cref{thm::main}, it follows that if $\ell$ is sufficiently large and the smallest prime divisor of $r / \gcd (r,\l)$ is $3$, then we have $\gamma (\C) = 1/3$. For example, for $r = 6$, the result of Han, Lo, and Sanhueza-Matamala~\cite{han2021covering} gives $\gamma (\mC^6_\l) = 1/2$ for every $\l \ge 72$ congruent to $1$, $3$, or $5 \pmod 6$. We have now that $\gamma (\mC^6_\l) = 1/3$ for every sufficiently large $\l$ congruent to $2$ or $4 \pmod 6$. 
Also, for $r = 3^m$, we have $\gamma(\C)= 1/3$ for every sufficiently large $\l$ not divisible by $r$.
These results reflect an interesting phenomenon that $\gamma(\C)$ can be very sensitive to the value of $\l$ and/or the factorization of $r$.

Our proof for \cref{thm::main} first considers the case where an infinite family containing $\C$ is forbidden. This approach was also used in~\cite{kamcev2024turan, balogh2024turan, sankar2024turan}. 
Given two $r$-graphs $\mH_1,\mH_2$, a \emph{homomorphism} from $\mH_1$ to $\mH_2$ is a map $f$ from $V(\mH_1)$ to $V(\mH_2)$ which maps $r$-edges to $r$-edges. 
We call $\mH_2$ a \emph{homomorphic image} of $\mH_1$ if $f$ induces a surjection from $E(\mH_1)$ to $E(\mH_2)$.
For a family $\mF$ of $r$-graphs, let $\mF\texthom$ denote the family of all homomorphic images of $r$-graphs in $\mF$.
Let $\Ckr$ be the (infinite) family of $r$-uniform tight cycles with length congruent to $k$ modulo $r$.
In~\cite{sankar2024turan}, Sankar gave a description of $\Ckr$-hom-free $r$-graphs using the language of group theory, generalizing the well-known description of $\mC_{\equiv 1}^2$-hom-free graphs as bipartite. Roughly speaking, 
an $r$-graph is $\Ckr$-hom-free if and only if an ``oriented coloring'' can be given to the $(r-1)$-tuples of vertices such that the colorings of the $(r-1)$-subtuples of any $r$-edge satisfy a certain ``consistency'' requirement. We discuss this in detail in \cref{sec::ckrfree}. Using this description of $\Ckr$-hom-free $r$-graphs, we prove that $\gamma (\Ckrh) \le 1/3$ if $r/\gcd(r,k)$ is odd, see \cref{pro::main}.
Finally, a simple and standard observation 
(based on deep theorems in extremal combinatorics) shows that $\gamma(\C) = \gamma(\Ckrh)$ for sufficiently large $\l \equiv k \pmod r$, see \cref{lem::codEq}.

The rest of this paper is organized as follows. 
In \cref{sec::conFree}, we show that the hypergraph in \cref{construction} is $\C$-free if $p$ divides $r/\gcd(r,\ell)$. 
In \cref{sec::ckrfree}, we explain the description of $\Ckrh$-free $r$-graphs. 
In \cref{sec::proof}, we give the proof of \cref{thm::main}, and in \cref{sec::conRem}, we have some concluding remarks.

\paragraph{Notation.} For a positive integer $n$, we write $[n]$ for the set $\{1,\ldots,n\}$.
For a set $S$ and an element $v$, we write $S + v$ for the set $S \cup \{v\}$ and $S - v$ for the set $S \sm \{v\}$.
For a set $S$ and a positive integer $i$, 
an $i$-set of $S$ is a subset of $S$ of size $i$, and we
let $\b{S}{i}$ be the family of  $i$-sets of $S$.
For a hypergraph $\mH = (V(\mH), E(\mH))$, we often use $\mH$ for its edge set, in particular, $|\mH|$ for  $|E(\mH)|$.
Given a hypergraph $\mH$ and sets of vertices $S,X$, let $N_X(S) \ce \{T \subseteq X \sm S : \ T \cup S \in E(\mH)\}$ be the \emph{link graph} of $S$ and $d_X(S) \ce |N_X(S)|$ be the \emph{degree} of $S$, respectively. When $X = V(\mH)$, we often  omit the subscript $X$ and just use $N(S)$ and $d(S)$.
For two disjoint sets $X,Y$, we say $T$ is an $(i,j)$-{\it set} of $(X,Y)$ if $T \subseteq X\cup Y$, $|T\cap X| = i$, and $|T \cap Y| = j$. If $X,Y$ are subsets of the vertex set of some $r$-graph $\mH$, then an $(i,j)$-{\it edge} is an $(i,j)$-set that is an edge of $\mH$.

\begin{remark}
    During the preparation of this manuscript, we learned that Ma and Rong~\cite{ma2025codegree}
    independently proved similar results, using a different  method. Moreover, they provided an additional family
    of constructions that yields a better bound than \cref{construction} for many pairs $(r,k)$.
\end{remark}

\section[Construction 2 is Clr-free]{\cref{construction} is $\C$-free} \label{sec::conFree}

In this section, we prove that the $r$-graph in \cref{construction} is $\C$-free for every $\ell$ with $\gcd(r,\ell)\mid r/p$.
In~\cite{han2021covering}, it was proved that if $p$ divides $r$ but does not divide $k$, then the $r$-graph in \cref{construction} is free of $\C$ for every $\l$ congruent to $k$ modulo $r$, see their Proposition 10.2. \cref{lem::conCkrhfree} says that it is sufficient to only require that $p$ divides $r / \gcd(r,\l)$.

\begin{lemma} \label{lem::conCkrhfree}
    Given a uniformity $r$ and some residue $k \not \equiv 0 \pmod r$, let $p$ be a divisor of $r / \gcd(r,k)$ and $\mH$ be the $r$-graph given by \cref{construction}. Then, $\mH$ contains no member of $\Ckrh$.
\end{lemma}
\begin{proof}
     Suppose for contradiction that $\mH$ contains a copy of $\C$-hom for some $\l \equiv k \pmod r$. Let $x_1,\ldots,x_\l$ be the (not necessarily distinct) vertices along such a copy, and for $1\leq j\leq \ell$, let $V_{i_j}$ be the part containing $x_j$. Hereafter, all subscripts containing $j$ are taken modulo $\ell$.
     
    We first prove that $i_j = i_{j + \gcd(r,k)}$ for every $j$.
    It suffices to prove that $i_j = i_{j+r}$ and $i_j= i_{j+k}$ for every $j$. By construction, for every $(r-1)$-set of vertices, its neighborhood is contained in a single set $V_i$. Due to edges $\{x_j,x_{j+1},\ldots, x_{j+r-1}\}$ and $\{x_{j+1},\ldots, x_{j+r-1},x_{j+r}\}$, we have $i_j = i_{j+r}$. Also, we trivially have $i_j = i_{j+\l}$. Because $\l \equiv k \pmod r$ and $i_j = i_{j+r}$, we then have $i_j = i_{j+k}$.
    
    Since $i_j=i_{j+\gcd(r,k)}$ for each $j$, we have that
    \[
    i_1+i_2+\ldots+i_r=\frac r{\gcd(r,k)}\left(i_1+\ldots+i_{\gcd(r,k)}\right).
    \]
    Because $p$ divides $r/\gcd(r,k)$, this sum is 0 modulo $p$. However, $\{x_1,\ldots,x_r\}\in E(\mH)$, so this sum is 1 modulo $p$ by the construction, yielding a contradiction.
\end{proof}

\section[Ckr-hom-free hypergraphs]{$\Ckr$-hom-free hypergraphs} \label{sec::ckrfree}

As mentioned in \cref{sec::Int}, our proof for \cref{thm::main} relies on the analysis of $\Ckr$-hom-free hypergraphs. 
In this section, we describe the $\Ckr$-hom-free $r$-graphs. We first give a brief introduction to the entire theory and then focus on the parts we need for the proof of \cref{thm::main}, see \cref{onlySi}.

\subsection{Permutations and oriented colorings}

For every uniformity $r\geq 2$ and each residue $k$ modulo $r$, Sankar~\cite{sankar2024turan} gave an equivalent description of $\Ckr$-hom-free $r$-graphs as those admitting a certain type of ``oriented coloring''. To state the characterization, we need some group theory.
 
An \emph{oriented edge} of an $r$-graph $\mH$ is an ordered $r$-tuple $\vx=x_1\dots x_r$ whose support $\{x_1,\ldots,x_r\}$ is an edge of $\mH$. Let $\EdH$ be the set of all oriented edges of $\mH$. There is a natural action of the symmetric group $S_r$ on $\EdH$ given by
\[
    \pi(x_1\ldots  x_r) = x_{\pi^{-1}(1)}\ldots    x_{\pi^{-1}(r)}
\]
for each permutation $\pi\in S_r$ and oriented edge $\vx=x_1\dots x_r\in\EdH$. 

An \emph{$S_r$-set} is a set equipped with a group action by $S_r$. One example is the set $\EdH$ of oriented edges of an $r$-graph $\mathcal{H}$ mentioned above. Another example is, for a subgroup $\Gamma\subseteq S_r$, the family $S_r/\Ga$ of left cosets $\sg\Ga:=\{\sg\ga:\ga\in\Ga\}$, where the action of $\pi\in S_r$ maps a coset $\sg\Ga\in S_r/\Ga$ to $(\pi\sg)\Ga$. If $A$ and $B$ are $S_r$-sets, a map $\chi:A\to B$ is called \emph{$S_r$-equivariant} if it preserves the structure of the $S_r$ action, i.e., if $\pi(\chi(a))=\chi(\pi(a))$ for each $a\in A$ and $\pi\in S_r$. With this in mind, we may define oriented colorings.

\begin{definition}
Let $\mH$ be an $r$-graph. An \emph{oriented coloring} of $\mH$ by an $S_r$-set $A$, also called an \emph{$A$-coloring}, is an $S_r$-equivariant map $\chi:\EdH\to A$. Each orbit of $A$ is called a \emph{color}, and the color of an edge $e\in E(\mH)$ is the orbit containing $\chi(\vec x)$ for each orientation $\vec x$ of $e$.
\end{definition}

To see why this might be considered a coloring, l  et $r=3$ and consider the set $A=\left\{\pointedtriangle,\pointedtriangleleft,\pointedtriangleright\right\}$ of pictograms. Then $S_3$ acts on $A$ by permuting the three vertices of the triangle. If $\chi$ is an $A$-coloring with $\chi(x_1x_2x_3)=\pointedtriangle$ for some oriented edge $\vx=x_1x_2x_3$ then
\[\begin{array}{c}
\chi(x_2x_1x_3)=(12)\chi(\vx)=\pointedtriangleleft,\quad
\chi(x_3x_2x_1)=(13)\chi(\vx)=\pointedtriangleright,\quad
\chi(x_1x_3x_2)=(23)\chi(\vx)=\pointedtriangle,\quad
\\[\medskipamount]\chi(x_3x_1x_2)=(123)\chi(\vx)=\pointedtriangleleft,\quad\text{and}\quad
\chi(x_2x_3x_1)=(132)\chi(\vx)=\pointedtriangleright.\vspace{-3pt}
\end{array}\]
One may reinterpret this as the edge $\{x_1,x_2,x_3\}$ being colored by $A$ in the orientation\pointedtrianglelabeled{x_1}{x_2}{x_3}. In fact, every $S_r$-set is isomorphic to a set of rotations/reflections of coloring(s) of the $(r-1)$-simplex. A more complicated example is the $S_3$-set $A'=\left\{\pointedtriangle,\pointedtriangleleft,\pointedtriangleright,\pinwheel,\pinwheelrev\right\}$, which contains two colors corresponding to the pictograms $\pointedtriangle$ and $\pinwheel$ (up to rotation/reflection).

\begin{definition}
Let $\chi:\EdH\to A$ be an oriented coloring of an $r$-graph $\mH$. We say that $\chi$ is \emph{accordant} if, for any two oriented edges $\vx=x_1\dots x_r$ and $\vx'=x_1\dots x_{i-1}x'_ix_{i+1}\dots x_r$ in $\EdH$ differing by exactly one (arbitrary) coordinate, it holds that $\chi(\vx)=\chi(\vx')$.	
\end{definition}

Pictorially, a coloring is accordant if, for any set $e'=\{x_1,\ldots,x_{r-1}\}$, all edges of the form $e'\cup\{y\}$ receive the same coloring in the same orientation. Using the $3$-uniform example $A'$ above, in an accordant $A'$-coloring, two edges $\{x_1,x_2,y\}$ and $\{x_1,x_2,y'\}$ overlapping on two vertices must be colored as either
\newcommand{\makepointeddiamond}[3]{
\!\begin{tikzpicture}[scale=0.8, every node/.append style={scale=0.8}, baseline=0.78em]
\path (30:1) coordinate (y2) -- (0,0) coordinate (x2) -- (0,1) coordinate (x1) -- (150:1) coordinate (y1);
\foreach \i in {1,2}
	\tikzshiftedpointedtriangle{#1}{#2}{#3}{(0,0)}{0.18}
\path (x1) --++(0,0.05) node {$x_1$};
\path (x2) --++(0,-0.05) node {$x_2$};
\path (y1) --++(-0.1,0) node {$y$};
\path (y2) --++(0.1,0) node {$y'$};
\end{tikzpicture}\!}
\newcommand{\makepindiamond}[3]{
\!\begin{tikzpicture}[scale=0.8, every node/.append style={scale=0.8}, baseline=0.78em]
\path (30:1) coordinate (y2) -- (0,0) coordinate (x2) -- (0,1) coordinate (x1) -- (150:1) coordinate (y1);
\foreach \i in {1,2}
	\tikzshiftedpinwheel{#1}{#2}{#3}{(0,0)}{0.18}
\path (x1) --++(0,0.05) node {$x_1$};
\path (x2) --++(0,-0.05) node {$x_2$};
\path (y1) --++(-0.1,0) node {$y$};
\path (y2) --++(0.1,0) node {$y'$};
\end{tikzpicture}\!}
\makepointeddiamond{x1}{x2}{y\i},
\makepointeddiamond{x2}{x1}{y\i},
\makepointeddiamond{y\i}{x1}{x2},
\makepindiamond{x1}{x2}{y\i},
or\makepindiamond{x2}{x1}{y\i}.
Accordance is harder to visualize when $r\geq 4$, but admits a simple combinatorial description for particularly nice $S_r$-sets, see~\cite{sankar2024turan,sankar2025topological}.

Let us now state \cref{coloringLemma}, which uses accordant colorings to characterize $\Ckr$-hom-free $r$-graphs. To streamline the 
statement of the theorem, we need one more piece of terminology.
For a permutation $\pi\in S_r$ and a subgroup $\Ga\subseteq S_r$, we say $\Ga$ is \emph{$\pi$-conjugate avoiding} if $\Ga$ contains no conjugate $\sg\pi\sg^{-1}$ of $\pi$. Let $\fG_\pi$ be the set of all maximal $\pi$-conjugate avoiding subgroups of $S_r$.
Note that $\fG_\pi$ can be partitioned as $\bigsqcup_{j=1}^m [\Ga_j]$, where $\Ga_1,\ldots, \Ga_m$ are subgroups of $S_r$ and $[\Ga_j]$ denotes the family of all conjugates of $\Ga_j$ in $S_r$.
Define $A_\pi \ce \bigsqcup_{j=1}^m (S_r / \Ga_j)$, recalling that $S_r/\Ga_j$ is the $S_r$-set of left cosets of $\Ga_j$ acted on by left multiplication by $S_r$. In this setting, we find it notationally simpler to refer to colors $S_r/\Ga_i$ by their associate groups $\Ga_i$.

\begin{theorem}[Theorem 3.4 in~\cite{sankar2024turan}] \label{coloringLemma}
 Fix a uniformity $r \ge 2$ and a residue $k$ modulo $r$. Let $\cyc=(12\dots r)\in S_r$ denote the cyclic shift permutation. We have that an $r$-graph $\mH$ is $\Ckr$-hom-free if and only if there is an accordant $A_{\cyc^k}$-coloring $\chi: \EdH \to A_{\cyc^k}$.
\end{theorem}

\begin{example}
When $(r,k)=(2,1)$, the $S_2$-set $A_{(12)}$ is isomorphic to the set
$\{
\tikz[scale=0.8, baseline=-0.1em]{
\draw[fill=black!70] (0,0) rectangle (-0.5,0.1);
\draw (0,0) rectangle (0.5,0.1);
},
\tikz[scale=0.8, baseline=-0.1em]{
\draw (0,0) rectangle (-0.5,0.1);
\draw[fill=black!70] (0,0) rectangle (0.5,0.1);
}
\}$ acted on by permuting the two endpoints. Then, \cref{coloringLemma} says that a graph $H$ is $\mC_{\equiv 1}^2$-hom-free if and only if one can color each edge of $H$ half white and half black such that either all or none of the  edges incident to a given vertex are black at that vertex. This corresponds to a black--white bipartition of the graph.
\end{example}

For a uniformity $r\geq 2$ and a residue $k$ modulo $r$, we say that the colors \emph{available} for $\Ckrh$ are the $\cyc^k$-conjugate avoiding groups $\Ga_j$ defined above.
Given an oriented coloring of an $r$-graph $\mH$ and an edge $x = \{x_1,\ldots,x_r\} \in \mH$, we say $x$ is colored with $\Ga$ if $\chi(x_1\ldots x_r) \in S_r / \Ga$; the color $\Ga$ is independent of the ordering of the edge.
For a $\Ckrh$-free $r$-graph $\mH$, when we say that $\mH$ is colored, we mean that we fix a map $\chi$ from $\EdH$ to $A_{\cyc^k}$, and we say a $\Ckrh$ available color $\Ga$ is \emph{used} in $\mH$ if at least one edge in $\mH$ is colored with $\Ga$.

\subsection [Accordant colorings by Si times Sr-i] {Accordant colorings by $S_i\times S_{r-i}$}

The coloring described in \cref{coloringLemma} is easier to visualize when the only groups $\Ga_j$ used are of the form $S_i\times S_{r-i}$ for some $1\leq i<r$. In this case, the $S_r$-set $S_r/(S_i\times S_{r-i})$ is isomorphic to the set of rotations or reflections of an $(r-1)$-simplex colored so that some set of  $i$ vertices is interchangeable and the remaining set of $r-i$ vertices is also interchangeable. Examples include\pointedtriangle\ for $S_3/(S_2\times S_1)$, \pointedtetrahedron\  for $S_4/(S_3\times S_1)$ and\edgetetrahedron\ for $S_4/(S_2\times S_2)$.

Consider the $S_r$-set $A=\bigsqcup_{i=1}^{r}\big(S_r/(S_i\times S_{r-i})\big)$ comprising these colors, and fix an $r$-graph $\mH$ equipped with an accordant $A$-coloring $\chi:\EdH\to A$.
Given an $r$-edge $e=\{x_1,\ldots,x_r\}$ colored by $S_i\times S_{r-i}$, define the partition $e=\hp i(e)\sqcup\hm i(e)$ so that $\hp i(e)$ is the set of $i$ interchangeable vertices and $\hm i(e)$ is the set  of $r-i$ interchangeable vertices. In particular, if $\vx=x_1\dots x_r$ and $\vy$ is obtained from $\vx$ by permuting vertices within each part $\hpm i(e)$, then $\chi(\vx)=\chi(\vy)$.
Additionally, if edge $e'=\{x_1,\ldots,x_{r-1},x'_r\}$ intersects $e$ in $r-1$ vertices, then accordance implies that
\begin{equation}
\label{eq:accordance-h-pm}
\hp i(e')=\begin{cases}
\hp i(e)&\quad  \text{if }\ x_r\notin\hp i(e),
\\\hp i(e)\cup\{x'_r\}-\{x_r\}&\quad \text{if }\ x_r\in\hp i(e),
\end{cases}
\end{equation}
and the same for $\hm i$. Note that being colored with $S_i \times S_{r-i}$ is the same as being colored with $S_{r-i} \times S_i$, while $h_i^\pm = h_{r-i}^\mp$.

Technically, if $i=r/2$, then the choice of $\hp i(e)$ and $\hm i(e)$ is not uniquely determined, as both parts have the same size. Nevertheless, accordance still lets us choose partitions $\hpm i(e)$ satisfying (\ref{eq:accordance-h-pm}) for any two edges $e,e'$ intersecting in $r-1$ vertices. For example, if $r=4$ and $i=2$, then $\hp 2$ could always correspond to the two vertices on the yellow edge of\edgetetrahedron\ and $\hm 2$ to the red edge.

We use similar notation for any $(r-1)$-set $W$: fix an arbitrary edge $e$ containing $W$, and say the color of $W$ is the color of $e$ and define $h_i^+(W) = h_i^+(e) \cap W$ and $h_i^- (W) = h_i^-(e) \cap W$. By the definition of being accordant, $h_i^+(W)$ and $h_i^-(W)$ are independent of the edge $e$, see \cref{fig::siedges}.

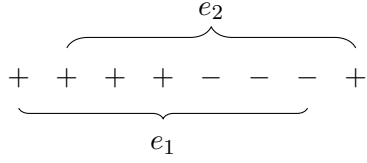
\begin{figure}[ht]
\centering
\begin{tikzpicture} [scale = 0.8]
  \node at (0.0,0) {$+$};
  \node at (0.8,0) {$+$};
  \node at (1.6,0) {$+$};
  \node at (2.4,0) {$+$};
  \node at (3.2,0) {$-$};
  \node at (4.0,0) {$-$};
  \node at (4.8,0) {$-$};
  \node at (5.6,0) {$+$};

  \draw[decorate,decoration={brace,mirror,amplitude=4pt}]
    (0.0,-0.5) -- (4.8,-0.5)
    node[pos=0.5,below=7pt] {$e_1$};

  \draw[decorate,decoration={brace,amplitude=8pt}]
    (0.8,0.5) -- (5.6,0.5)
    node[pos=0.5,above=7pt] {$e_2$};
\end{tikzpicture}
\caption{Two edges $e_1,e_2$ colored with $S_4\times S_3$ in a $7$-graph.  The vertices $+$ and $-$ are the vertices in $h_i^+$ and $h_i^-$, respectively. The shared $6$-tuple is of the same coloring in these two edges. Note that the signs $\pm$ represent colorings within $e_1$ and $e_2$; if there were an edge $e_3$ containing the rightmost five vertices and two extra vertices, then $e_3$ is not necessarily even colored with $S_4 \times S_3$.}
\label{fig::siedges}
\end{figure}

\subsection[Restricting to colors of the form Si times Sr-i] {Restricting to colors of the form $S_i\times S_{r-i}$}

The previous subsection shows that the accordant oriented coloring described in \cref{coloringLemma} has a relatively straightforward combinatorial description if the only groups $\Ga_j$ used are of the form $S_i\times S_{r-i}$. The first key step in our proof of \cref{thm::main} is to show that these are the only groups used when coloring $\Ckrh$-free $r$-graphs $\mH$ of large minimum codegree.

\begin{lemma} \label{onlySi}
Let $\mH$ be an $n$-vertex $\Ckrh$-free $r$-graph with $\delta_{r-1}(\mH)>n/3$ and let $\chi:\EdH\to A_{\cyc^k}$ be the accordant oriented coloring guaranteed by \cref{coloringLemma}. If $\Ga$ is used in $\mH$, then $\Ga$ is a conjugate of $S_i\times S_{r-i}$ for some $1\leq i<r$.
\end{lemma}
\begin{proof}
    Let $\Gatr$ be the subgroup of $\Ga$ generated by all transpositions in $\Ga$.
    We first prove that $\Gatr \cong S_i \times S_{r-i}$ for some $i \in [r-1]$. 
    
    Define the relation $\sim_\Ga$ on $[r]$ by $a \sim_\Ga b$ if $(ab) \in\Ga$; this is transitive: if $(ab), (bc) \in \Ga$, then $(ac) = (bc)(ab)(bc)$ is also in $\Ga$. Thus, this relation partitions $[r]$ into sets $I_1\cup\ldots \cup I_m$ and hence $\Ga$ contains $S_{|I_1|} \times \ldots \times S_{|I_m|}$. 
    Clearly $m\geq 2$, or else $\Ga=S_r$ would contain $\cyc^k$.
    
    Next, we use the minimum codegree assumption to show $m=2$. Suppose for contradiction that $m \ge 3$ and assume without loss of generality that $1\in I_1$, $2\in I_2$, and $3\in I_3$. Let $\vx = x_1\ldots x_r$ be an oriented edge in $\mH$ with $\chi(\vx)=\Ga\in S_r/\Ga$ (viewing $\Ga$ as a left coset of itself). For $j=1,2,3$, define $X_j=N(\{x_1,\ldots,x_r\}-\{x_j\})$; note that $x_j \in X_j$.
    We claim that $X_1,X_2,X_3$ are pairwise disjoint.
    Indeed, if e.g.~there was $v \in X_1 \cap X_2$, then
    \[
        (12)\chi(\vx) =
        (12)\chi(vx_2\ldots x_r) =
        \chi((12)(vx_2\ldots x_r)) =
        \chi(x_2v \ldots x_r) = 
        \chi(x_1v \ldots x_r) = 
        \chi (\vx),
    \]
    where the second-to-last equality uses accordance. This would yield $(12)\Ga=\Ga$ and hence $(12) \in \Ga$, contradicting the assumption that $1\not\sim_\Ga 2$. Thus, the sets $X_1,X_2,X_3$ are pairwise disjoint, yielding the contradiction
    \[
        n = |V(\mH)| \ge |X_1| + |X_2| + |X_3| \geq 3\cdot \del_{r-1}(\mH) > 3 \cdot n/3 = n.
    \]
    Thus, $m=2$. Setting $I=I_1$ and $J=I_2$, we may write $\Gatr=S_I\times S_J$.
    
    \begin{claim} \label{Simaximal}
    Either $\Ga=\Gatr$ or $\Ga$ contains an $r$-cycle.	
    \end{claim}
    
    \begin{proof}
    For the proof of this claim only, we will use the notation $[a,b]\ce (ab)$ to denote the permutation transposing $a$ and $b$. 
    
    Suppose $\Gatr$ is a proper subgroup of $\Ga$.
    We first prove that $\Ga$ contains a permutation mapping $I$ to $J$ bijectively, which in particular implies that $|I|=|J|=r/2$. Assume without loss of generality that $|I|\geq|J|$. Choose $\pi\in\Ga-(S_I\times S_J)$; it follows that $\pi$ maps some $a\in I$ to $b\ce\pi(a)\in J$. For any other $a'\in I$, observe that
    \[[\pi(a'),b]=[\pi(a'),\pi(a)]=\pi\circ[a,a']\circ\pi^{-1}\in\Ga.\]
    Thus, $\pi(a')\sim_\Ga b$ so $\pi(a')\in J$. It follows that $\pi$ maps $I$ injectively into $J$; moreover, this is a bijection because $|I|\geq|J|$.
    
    Let $i=r/2$ and assume without loss of generality that $I=\{1,\ldots,i\}$ and $J=\{i+1,\ldots,r\}$,
    and that $\Ga$ contains the permutation $\pi$ mapping $a$ to $a + i$ for every $a \in I$. Let $g$ be the cycle $(12\ldots i)$. Then, we have
    \[
        \pi g = (1(2+i)2(3+i)3 \ldots (2i)i(1+i)),
    \]
    so $\Ga$ contains an $r$-cycle.
    \end{proof}
    
    By \cref{Simaximal}, either $\Ga=\Gatr\cong S_{|I|}\times S_{|J|}$ or $\Ga$ contains an $r$-cycle $\pi'$. However, the latter case is impossible as $(\pi')^k$ would be a conjugate of $\cyc^k$.\qedhere
\end{proof}

\section[Proof of the main theorem]{Proof of the main theorem} \label{sec::proof}
In this section, we prove \cref{thm::main}. 
We first observe (\cref{lem::codEq}) that $\gamma (\C) = \pich$ for sufficiently large $\l$ congruent to $k$ modulo $r$. Then, we can focus on $\Ckrh$-free $r$-graphs and apply the coloring results for $\Ckrh$-free $r$-graphs discussed in \cref{sec::ckrfree}. 
\cref{thm::main} follows easily from \cref{pro::main}, a characterization of the pairs $(r,k)$ such that $\pich > 1/3$.

\cref{lem::codEq} connects the codegree Tur\'an densities of $\C$ and $\Ckrh$. 
We include the proof here for completeness, but we note that this was essentially proved in \cite{sankar2024turan} (see Propositions 2.3 and 4.1 and Corollary 4.2) and presented in a more general form in \cite{sankar2025topological} (see Theorem 7.3.3 and Lemma 7.5.2). For the case $r=3$, this was also essentially proved in \cite{kamcev2024turan} (see Theorem 6.2 and Propositions 6.3--6.4). 
\begin{lemma}\label{lem:supersaturation}
Let $\mF$ be any finite family of $r$-graphs. Then $\ga(\mF)=\ga(\mF\texthom)$.
\end{lemma}

\begin{proof}
Given an $r$-graph $\mH$, its \emph{$t$-blowup} $\mH[t]$ is formed by replacing each vertex of $\mH$ by an independent set of size $t$ and replacing each edge of $\mH$ with a copy of the complete $r$-partite $r$-graph $K_{t,\ldots,t}^{(r)}$. Write $\mF[t]=\{\mH[t]:\mH\in\mF\}$. Observe that if $\mH'\in\mF\texthom$ is a homomorphic image of $\mH\in\mF$, then $\mH\subseteq\mH'[t]$ for any $t\geq|V(\mH)|$. Thus, for $t_0=\max_{\mH\in\mF}|V(\mH)|$, we have 
$\ga (\mF) \le \ga(\mF\texthom[t_0])$.
Also, a well-known fact is that the codegree Tur\'an density of any finite family is the same as the codegree Tur\'an density of its blow-up, see Lemma 2.3 in~\cite{keevash2007codegree}, so we have $ \ga (\mF) \le \ga(\mF\texthom[t_0])= \ga(\mF\texthom)$.
 
The other direction $\ga(\mF\texthom)\le \ga(\mF)$ holds by definition.
\end{proof}

\begin{lemma} \label{lem::codEq}
    For every uniformity $r \ge 2$ and residue class $k \not\equiv 0 \pmod r$, there is an integer $L$ such that for every $\l \ge L$ with $\l \equiv k \pmod r$, we have $\gamma (\C) = \pich$. 
\end{lemma}
\begin{proof}
    Let $L$ be sufficiently large and fix an integer $\l \ge L$ with $\l \equiv k \pmod r$.
    As $\C \in \Ckrh$, we have $\pich \le \gamma(\C)$. By \cref{lem:supersaturation}, it suffices to prove that $\gamma(\C\texthom) \le \pich$.

    Let $\CkLr$ be the family of tight cycles whose length is congruent to $k$ modulo $r$ and is less than $L$. We first claim that $\gamma(\CkLrh) = \pich$. Indeed, by \cref{construction} and \cref{lem::conCkrhfree}, we have $\pich \ge \gcd(r,k)/r \ge 1/r$. Then, by the proofs of Proposition 4.1 and Corollary 4.2 in~\cite{sankar2024turan}, any $n$-vertex $\CkLrh$-free $r$-graph $\mH$ with $\delta_{r-1}(\mH) \ge n/r$ is also $\Ckrh$-free, given that $L$ is sufficiently large with respect to $r$. Hence, we have $\gamma(\CkLrh) \le \pich$.

    Then, by Proposition 2.3 in~\cite{sankar2024turan}, we have that every $\C$-hom-free $r$-graph is also $\CkLrh$-free, given $\l \ge L$ and $\l \equiv k \pmod r$, so $\gamma(\Ch) \le \gamma(\CkLrh)$.
\end{proof}

\cref{thm::main} is now a direct corollary of the following slightly stronger proposition, which cha\-racterizes the pairs $(r,k)$ such that $\gamma(\Ckrh) > 1/3$. 
\begin{theorem} \label{pro::main}
Fix a uniformity $r \ge 2$ and a residue class $k \not\equiv 0 \pmod r$.
If $r$ is even and $\sit$ is available for $\Ckrh$ for every odd $i \in [1,r-1]$, then $\gamma(\Ckrh) = 1/2$.
Otherwise, $\gamma(\Ckrh) \le 1/3$.    
\end{theorem}
\begin{proof}[Proof of \cref{thm::main} assuming \cref{pro::main}]
    Given the uniformity $r$ and residue class $k$, by Lem\-ma~\ref{lem::codEq}, it suffices to prove that $\gamma(\Ckrh) \le 1/3$, and by \cref{pro::main}, we just need to prove that it cannot be the case that $r$ is even and $\sit$ is available for every odd $i \in [1,r-1]$. Suppose for a contradiction, that it is the case.
    Write $m$ for $\gcd(r,k)$; by our assumption in \cref{thm::main}, $r/m$ is odd.
    Note that $\cyc^k$ consists of $m$ cycles, each of which has length $r/m <r$, and hence, $S_i \times S_{r-i}$ contains a conjugate of $\cyc^k$ if and only if $r/m \mid i$ and $r/m \mid r-i$. 
    Now, since $r/m \mid r/m $ and $r/m \mid (r-r/m)$, we have that $S_{r/m} \times S_{r-r/m}$ contains a conjugate of $\cyc^k$, a contradiction to \cref{coloringLemma}. 
    \qedhere
\end{proof}

In the remaining part of this section, we  prove \cref{pro::main}. The main idea is to analyze  the structure of the link graphs of some $(r-2)$-sets, which will give a partition of all the vertices.
A sketch of the proof will also be given at the beginning of our proof for \cref{pro::main}.
We first use this idea to prove an easy upper bound $1/2$, which will be used later. 
Recall that for every edge $e$ and integer $i \in [1,r-1]$, $h_i^+(e)$ is the set of vertices $v \in e$ of $S_i$ and $h_i^- (e)$ is the set of vertices $v \in e$ of $S_{r-i}$. 

\begin{lemma} \label{le1over2}
    For every uniformity $r \ge 2$ and residue $k \not\equiv 0 \pmod r$, we have
    $ \gamma(\Ckrh) \le 1/2$.
\end{lemma}
\begin{proof}
    Suppose for a contradiction that $\pich > 1/2$. Let $\eps > 0$ and $n$ be a sufficiently large integer, and assume that $\mH$ is an $n$-vertex $\Ckrh$-free colored $r$-graph with $\drmh > (1/2+\eps)n$. 

    Fix an arbitrary edge $e$ in $\mH$.
    By \cref{onlySi}, we can assume that $e$ is colored with $S_i \times S_{r-i}$ for some $i \in [1,r-1]$.
    We assume that $e = \{a_1,\ldots, a_i, b_1,\ldots,b_{r-i}\}$, where
    $A \ce \{a_1,\ldots, a_{i}\} = h_i^+(e)$, and $B \ce \{b_1,\ldots,b_{r-i}\} = h_i^-(e)$. 
    Let $A' \ce \{a_1,\ldots,a_{i-1}\}$, $B' \ce \{b_1,\ldots,b_{r-i-1}\}$, and let $U \ce A' \cup B'$. We consider the coloring of each vertex $v \in V(\mH) \sm U$ in $U + v$.
    Define 
    \begin{align*}
        X^+ & \ce A \cup 
        \{v \in V(\mH) \sm e: h_i^+(U+v) = A' + v  \  \textrm{ and }\  h_i^-(U+v) = B'\}, \\
        X^- & \ce B \cup 
        \{v \in V(\mH) \sm e: h_i^+(U+v) = A' \ \textrm{ and }\  h_i^-(U+v) = B' + v \},
    \end{align*}
    see \cref{fig::le1over2}.
    Note that $X^+$ and $X^-$ are disjoint, $a_i \in X^+$ and $b_{r-i} \in X^-$. Also, note that $N(U+a_i) \subseteq X^-$ by the coloring of $U + a_i$, so 
    \[
        |X^-| \ge d(U+a_i) \ge \drmh > (1/2+\eps)n.
    \]
    Similarly, we have $|X^+| > (1/2+\eps)n$. However, this is a contradiction to the fact that $X^+$ and $X^-$ are disjoint. \qedhere
\end{proof}
The following lemma will be used later to show that some color has to be used: it essentially says that if there are two disjoint sets in an $r$-graph with lots of $(i,r-i)$-edges, then the color $S_i \times S_{r-i}$ has to be used.
\begin{figure}[ht]
  \centering
  \resizebox{0.45\linewidth}{!}{
    \begin{tikzpicture}[line cap=round, line join=round, font=\small, yscale=0.8]
      \draw[thick] (0,0) rectangle (7,2.2);
      \node at (6.65,1.4) {$X^-$};

      \draw[thick] (0,2.8) rectangle (7,5.0);
      \node at (6.65,4.3) {$X^+$};

      \coordinate (a1) at (1.1,3.8);
      \coordinate (a2) at (2.2,3.8);
      \coordinate (a3) at (3.5,3.8);
      \coordinate (a)  at (5.5,3.8);
      \fill (a1) circle (1.6pt) node[above=1pt] {$a_1$};
      \fill (a2) circle (1.6pt) node[above=1pt] {$a_2$};
      \fill (a3) circle (1.6pt) node[above=1pt] {$a_3$};
      \fill (a)  circle (1.6pt);

      \coordinate (b1) at (1.1,0.8);
      \coordinate (b2) at (2.2,0.8);
      \coordinate (b3) at (3.5,0.8);
      \coordinate (b)  at (5.5,0.8);
      \fill (b1) circle (1.6pt) node[above=1pt] {$b_1$};
      \fill (b2) circle (1.6pt) node[above=1pt] {$b_2$};
      \fill (b3) circle (1.6pt) node[above=1pt] {$b_3$};
      \fill (b)  circle (1.6pt);

      \draw[thick, rounded corners=6pt] (0.45,0.2) rectangle (4,4.75);
      \node at (3.8,4.55) {$e$};

      \draw[thick, rounded corners=6pt] (0.85,0.4) rectangle (2.8,4.4);
      \node at (2.6,4.2) {$U$};
    \end{tikzpicture}
  }
  \caption{$r = 6$ and $i = 3$. The coloring of every vertex $v \in X^+ \sm e$ in $U+v$ is the same as $a_3$ in $e$. The coloring of every vertex $v \in X^- \sm e$ in $U+v$ is the same as $b_3$ in $e$.}
  \label{fig::le1over2}
\end{figure}
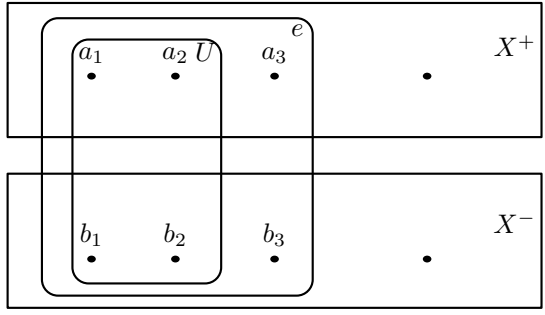

\begin{lemma} \label{manyToColoring}
    Let $\mH$ be an $n$-vertex $\Ckrh$-free colored $r$-graph in which all colors used are in $\{S_{t}\times S_{r-t}: t \in [r-1]\}$. 
    Let $X,Y$ be two disjoint subsets of $V(\mH)$ such that $|X|,|Y| \ge r$.
    Let $i$ be some fixed integer in $[r-1]$.
    Suppose that for every $(i-1,r-i)$-set $W$ of $(X,Y)$, we have $d_X(W) > |X|/2$.
    Then, $S_i \times S_{r-i}$ is used in $\mH$.
\end{lemma}
\begin{proof}
    We will prove a stronger statement that all  $(i,r-i)$-edges of $(X,Y)$ are colored with $S_i\times S_{r-i}$.
     Suppose for a contradiction that there is an $(i,r-i)$-edge $e$ of $(X,Y)$ such that $e$ is colored with $S_j \times S_{r-j}$ where $j \not\in\{ i, r-i\}$. Then, $h_j^+(e) \cap X$ and $h_j^-(e) \cap X$ are both non-empty or $h_j^+(e) \cap Y$ and $h_j^-(e) \cap Y$ are both non-empty.
    We claim that neither of them can happen.

    If there are vertices $x_1 \in h_j^+(e) \cap X$ and $x_2 \in h_j^-(e) \cap X$,
    then let $W_1 = e - x_1 $, $W_2 = e - x_2$, and $U = e - x_1 - x_2 = W_1\cap W_2$, see \cref{fig::lem13_1}. By our assumption, we have $d_X(W_1) > |X|/2$ and $d_X(W_2) > |X|/2$, so there is a vertex $v \in N_X(W_1) \cap N_X(W_2)$. Let $W = U + v$. Then, $x_1,x_2 \in N_X(W)$, so $W_1 = W - v + x_2$ and $W_2 = W - v + x_1$ should be of the same coloring, which in particular, means that either $x_1 \in h_j^+(W_2), x_2 \in h_j^+(W_1)$ or $x_1 \in h_j^-(W_2), x_2 \in h_j^-(W_1)$, but this is a contradiction to our assumption that $x_1 \in h_j^+(e)$ and $x_2 \in h_j^-(e)$.

    A similar argument holds for the second case.
    If there are vertices $y_1 \in h_j^+(e) \cap Y$ and $y_2 \in h_j^-(e) \cap Y$, then fix vertices $x \in e \cap X$, $y \in Y \sm e$ and let $W_0 = e - x$, $W_1 = e - x - y_1 + y$, $W_2 = e - x - y_2 + y$. Note that $y_1 \in h_j^+(W_0)$ and $y_2 \in h_j^-(W_0)$ by our assumption.
    Now, since $d_X(W_0) > |X|/2$ and $d_X(W_1) > |X| / 2$, there is a vertex $v_{01} \in N_X(W_0) \cap N_X(W_1)$ and we have edges $W_0 + v_{01}$ and $W_1 + v_{01}$. Note that both of them contain the same $(r-1)$-set $W_0 - y_1 + v_{01} = W_1 - y + v_{01}$, so $W_0$ and $W_1$ should be of the same coloring, in particular, we have $y_2 \in h_j^-(W_1)$ by the assumption that $y_2 \in h_j^-(W_0)$. Similarly, we have $y_1 \in h_j^+(W_2)$. However, since $d_X(W_1) > |X|/2$ and $d_X(W_2) > |X| / 2$,
    there is a vertex $v_{12} \in N_X(W_1) \cap N_X(W_2)$ and we have edges $W_1+ v_{12}$ and $W_2+v_{12}$. Hence, 
    we have either $y_1 \in h_j^+(W_2), y_2 \in h_j^+(W_1)$ or $y_1 \in h_j^-(W_2), y_2 \in h_j^-(W_1)$, a contradiction. \qedhere 
\end{proof}
\begin{figure}[ht]
    \centering
    \begin{tikzpicture}[scale=0.9, line cap=round, line join=round, font=\small, yscale=0.9]
    
      \draw[thick] (0,0) rectangle (7,2.2);
      \node at (6.65,1.1) {$Y$};
    
      \draw[thick] (0,2.8) rectangle (7,5.0);
      \node at (6.65,4.0) {$X$};
    
      \draw[thick, rounded corners=3pt] (0.9,0.4) rectangle (3.5,1.2);
      \node at (2.2,0.95) {$U$};
    
      \foreach \p in {(1.2,0.61),(2.2,0.61),(3.2,0.61)}{
        \fill \p circle (1.6pt);
      }
    
      \coordinate (x1) at (1.6,4.1);
      \coordinate (x2) at (2.8,4.1);
      \coordinate (v)  at (5.6,4.1);
    
      \fill (x1) circle (1.6pt) node[above=1pt] {$x_1$};
      \fill (x2) circle (1.6pt) node[above=1pt] {$x_2$};
      \fill (v)  circle (1.6pt) node[above=1pt] {$v$};
    
      \coordinate (uL) at (0.95,1.2);
      \coordinate (uR) at (3.45,1.2);
    
      \draw[thick] (uL) -- (x1);
      \draw[thick] (uR) -- (x2);
    
      \draw[thick] (uL) -- (x2);
      \draw[thick] (uR) -- (x1);
    
      \draw[thick] (uL) -- (v);
      \draw[thick] (uR) -- (v);
    
      \node at (1.75,3.1) {$W_2$};
      \node at (2.65,3.1) {$W_1$};
    
      \node at (4.6,3.2) {$W$};
    
      \draw[thick, rounded corners=6pt] (0.65,0.2) rectangle (3.75,4.75);
      \node at (3.55,4.55) {$e$};

    \end{tikzpicture}
        
    \caption{$r=5$ and $i=2$. Both $U+x_1+v$ and $U+x_2+v$ are edges in $\mH$. Hence, the coloring of $x_1$ in $U+x_1$ should be the same as the coloring of $x_2$ in $U+x_2$.}
    \label{fig::lem13_1}
\end{figure}
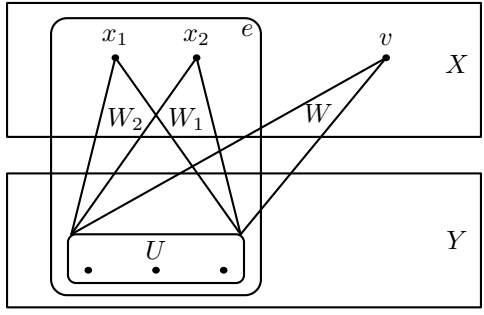

\begin{proof} [Proof of \cref{pro::main}]
    If $r$ is even and $\sit$ is available for every odd $i \in [1,r-1]$, then \cref{le1over2} gives that $\pich \le 1/2$.
    For the other direction, taking $p = 2$ in \cref{construction} works. More precisely, let $\mH$ be an $n$-vertex $r$-graph whose vertex set is partitioned into two parts $X_1,X_2$ of size $\lfloor n/2 \rfloor$, $\lceil n/2 \rceil$, respectively, and its edges are those $r$-sets which intersect both $X_1, X_2$ in an odd number of vertices. 
    Clearly, $\delta_{r-1}(\mH) = (1/2 - o(1)) n$.
    For an $(i,r-i)$-edge $e$ of $(X_1,X_2)$, we can color it with $S_i\times S_{r-i}$ where $h_i^+(e) = e \cap X_1$ and $h_i^-(e) = e \cap X_2$. Then by \cref{coloringLemma}, we have that $\mH$ is $\Ckrh$-free. This gives that $\gamma(\Ckrh) \ge 1/2$.

    Now, it suffices to prove that if $\gamma(\Ckrh) > 1/3$, then $r$ is even and $\sit$ is available for every odd $i \in [1,r-1]$.
    Let $\eps > 0$ and $n$ be sufficiently large. Assume that $\mH$ is an $n$-vertex $\Ckrh$-free colored $r$-graph with $\delta_{r-1}(\mH) > (1/3+\eps) n$. By \cref{onlySi}, we can assume that every color used in $\mH$ is of the form $S_i \times S_{r-i}$. 

    A sketch of the proof is as follows. Our main claim is that if some $S_i \times S_{r-i}$ is used in $\mH$, then $S_{i+2} \times S_{r-i-2}$ and $S_{i-2} \times S_{r-i+2}$ are also used, see \cref{itoi+2i-2}. For this purpose, we will choose an arbitrary edge colored with $S_i \times S_{r-i}$ and consider the link graph of a ``good'' $(r-2)$-set $U$ in $W$, like what we did in the proof of \cref{le1over2}. This will give us a partition $(X^+,X^-)$ of the vertex set of $\mH$. We then analyze the coloring of the $(i,r-i-1)$-sets and $(i-1,r-i)$-sets of $(X^+,X^-)$. By the coloring, we can verify that there is no $(i+1,r-i-1)$-edges and $(i-1,r-i+1)$-edges of $(X^+,X^-)$, which then forces the existence of many $(i+2,r-i-2)$-edges and $(i-2,r-i+2)$-edges, due to the minimum codegree assumption. Using \cref{manyToColoring}, we then get that $S_{i+2} \times S_{r-i-2}$ and $S_{i-2} \times S_{r-i+2}$ have to be used, as claimed. Additionally, a simple observation shows that $S_{r-2} \times S_2$ cannot be used in $\mH$, see \cref{nor-2}. Putting these two claims together, we immediately conclude that $r$ has to be even and $\sit$ is available for every odd $i \in [1,r-1]$. 

    An $(r-2)$-set $U$ is \emph{$i$-good} if there is an $(r-1)$-set $W \supset U$ such that $W$ is colored with $S_i \times S_{r-i}$, $|U \cap h_i^+(W)| = i - 1$, and $|U \cap h_i^-(W)| = r-i-1$; we say such $W$ \emph{witnesses} $U$. Recall that $h_i^{\pm}(W)$ is defined as $h_i^{\pm}(e)\cap W$ for any edge $e\supseteq W$. For every $i$-good $(r-2)$-set $U$, define
    \[
        Y^+_{i}(U) \ce \{v \in V(\mH) \sm U : v \in h^+_i(U+v) \}, \quad \quad 
        Y^-_{i}(U) \ce \{v \in V(\mH) \sm U : v \in h^-_i(U+v) \}.
    \]
    \begin{claim} \label{propertyY}
        For every $i$-good $(r-2)$-set $U$, we have 
        \[
            (1/3+\eps)n \le |Y^+_i(U)|, |Y^-_i(U)| \le (2/3 - \eps)n - (r-2),
        \]
        and $Y^+_i(U), Y^-_i(U)$ form a partition of $V(\mH) \sm U$.
    \end{claim}
    \begin{proof}
        Let $W$ be an arbitrary $(r-1)$-set witnessing $U$ and assume that $U = W - v$. We have $v \in Y^+_i(U) \cup Y^-_i(U)$. 
        For the case  $v \in Y^+_i(U)$, by the coloring, we have $N(W) \subseteq Y^-_i(U)$, so $|Y^-_i(U)| \ge \delta_{r-1}(\mH) \ge (1/3+\eps)n$. 
        Then, we can fix an arbitrary vertex $v' \in Y^-_i(U)$. We also have $N(U+v') \subseteq Y^+_i(U)$ by the coloring, so $|Y^+_i(U)| \ge (1/3+\eps)n$. Note that $Y^+_i(U)$ and $Y^-_i(U)$ are disjoint sets in $V(\mH) \sm U$. We then have $|Y^+_i(U)|, |Y^-_i(U)| \le (2/3 - \eps)n - (r-2)$. 
        For the case where $v \in Y^-_i(U)$, a similar proof holds by switching the $+$ and $-$.

        Let $Y' \ce V(\mH) \sm (U \cup Y^+_i(U) \cup Y^-_i(U))$. Suppose for a contradiction that there is $y \in Y'$. Then, by definition, $U+y$ is not colored with $S_i \times S_{r-i}$, so $N(U+y) \cap Y^+_i(U) =N(U+y) \cap Y^-_i(U) = \es$. Hence, $N(U+y) \subseteq Y'$, so $|Y'| \ge (1/3+\eps)n$. However, we have a contradiction that
        \[
            n = |V(\mH)| \ge |Y^+_i(U)| + |Y^-_i(U)| + |Y'| \ge 3 \cdot (1/3+\eps)n > n. \qedhere
        \]
    \end{proof}

    For the convenience of the following discussion, we will also partition $U$ and add it to $Y^+_i(U)$ and $Y^-_i(U)$. We need the following claim.
    \begin{claim} \label{parUindW}
    For every $i$-good $(r-2)$-set $U$ and any two $(r-1)$-sets $W_1,W_2$ witnessing $U$, we have 
    \[
        h_i^+(W_1)\cap U = h_i^+(W_2)\cap U \quad\quad\text{and}\quad\quad 
        h_i^-(W_1)\cap U = h_i^-(W_2)\cap U.
    \]
    \end{claim}
    \begin{proof}
        Assume that $U = W_1 - v_1 = W_2 - v_2$. First, consider the case $v_1,v_2\in Y^+_i(U)$. 
        Note that $N(W_1), N(W_2) \subseteq Y^-_i(U)$.
        By \cref{propertyY} and the assumption on $\delta_{r-1}(\mH)$, there is $v \in Y^-_i(U)$ such that $v \in N(W_1) \cap N(W_2)$, so by the coloring, $h_i^+(W_1)\cap U = h_i^+(W_2)\cap U = h_i^+(U+v) \cap U$ and $ h_i^-(W_1)\cap U = h_i^-(W_2)\cap U = h_i^-(U+v)\cap U$. Similar argument holds if $v_1,v_2 \in Y^-_i(U)$. If $v_1 \in Y^+_i(U)$ and $v_2 \in Y^-_i(U)$, then just fix an arbitrary vertex $v' \in N(W_2) \subseteq Y^+_i(U)$ and apply the proof for the case that $v_1,v_2$ are in the same set to $v_1,v'$.
    \end{proof}
    
    By \cref{parUindW}, for an $i$-good $(r-2)$-set $U$, we can fix an arbitrary $(r-1)$-set $W\supset U$ and define $h_i^+(U) \ce h_i^+(W) \cap U$, $h_i^-(U) \ce h_i^-(W) \cap U$. Let $X^+_i(U) \ce h_i^+(U) \cup Y^+_i(U)$ and $X^-_i(U) \ce h_i^-(U) \cup Y^-_i(U)$. We often omit the subscript $i$ when it is clear from the context. By \cref{propertyY}, we have the following properties of $X^+_i(U), X^-_i(U)$.
    \begin{claim} \label{propertyX}
        For every $i$-good $(r-2)$-set $U$, we have 
        \[
            (1/3+\eps)n \le |X^+_i(U)|, |X^-_i(U)| \le (2/3 - \eps)n,
        \]
        and $X^+_i(U), X^-_i(U)$ form a partition of $V(\mH)$.
    \qed
    \end{claim}
    
    Now, we analyze the colorings of $(i-1,r-i)$-sets and $(i,r-i-1)$-sets of $(X^+_i(U), X^-_i(U))$.

    \begin{claim} \label{colorInBetween}
        Let $U$ be an $i$-good $(r-2)$-set.
        Suppose $W$ is an $(i, r-i-1)$-set or an $(i-1,r-i)$-set of $(X^+(U), X^-(U))$. Then $W$ is colored with $S_i \times S_{r-i}$, and $h_i^+(W) = W \cap X^+(U)$ and $h_i^-(W) = W \cap X^-(U)$.
    \end{claim}
    \begin{proof}
        We first have the following claim.
    \begin{subclaim} \label{goodMoveOne}
        For every $i$-good $(r-2)$-set $U$, let $x$ be an arbitrary vertex in $h^+_i(U)$ and $v$ be an arbitrary vertex in $X^+(U) \sm U$. We have that $U - x + v$ is also $i$-good. Moreover, $X^+(U-x+v) = X^+(U)$, and $X^-(U-x+v) = X^-(U)$.
    \end{subclaim}

    \begin{proof}
        Let $U'$ be the $(r-2)$-set $U - x + v$ and $W$ be the $(r-1)$-set $U \cup U' = U + v$. \cref{propertyY} implies that $W$ is colored with $S_i \times S_{r-i}$, $|U \cap h_i^+(W)| = i - 1$, and $|U \cap h_i^-(W)| = r-i-1$. Then, by the assumption on $x$ and $v$, we have 
        $|U' \cap h_i^+(W)| = |U \cap h_i^+(W)| = i-1$ and $|U' \cap h_i^-(W)| = |U \cap h_i^-(W)| = r-i-1$. Therefore, $W$ witnesses $U'$, hence $U'$ is $i$-good. Note that $v \in h^+_i(U')$.

        By \cref{propertyX}, we have that $X^+(U')$ and $X^-(U')$ also form a partition of $V(\mH)$. Suppose for a contradiction that $X^+(U') \neq X^+(U)$ or $X^-(U') \neq X^-(U)$. Then, there is a vertex $y$ that is in $X^+(U) \cap X^-(U')$ or in $X^-(U) \cap X^+(U')$. Note that $y \notin U \cup U'$.
        Let $U'' \ce U - x + y = U' - v + y$. We first claim that $U''$ is $i$-good. Indeed, similarly to the proof that $U'$ is good, if $y \in X^+(U)$, then $U+y = U'' + x$ witnesses that $U''$ is good; if $y \in X^+(U')$, then $U' + y = U'' + v$ witnesses that $U''$ is good.
        Now, by \cref{parUindW}, $y \in U''$ should be of the same coloring in all $(r-1)$-sets containing $U''$, which, in particular, means that it is the same in 
        $U'' + x = U + y$ and in $U''+ v = U' + y$.
        However, by the assumption that either $y \in X^+(U) \cap X^-(U')$ or $y \in X^-(U) \cap X^+(U')$, we have that $y$ should be of different colors in these two sets, a contradiction. \qedhere
    \end{proof}

    An analogous argument yields the following claim.
    \begin{subclaim} \label{goodMoveOne-}
        For every $i$-good $(r-2)$-set $U$, let $x$ be an arbitrary vertex in $h^-_i(U)$ and $v$ be an arbitrary vertex in $X^-(U) \sm U$. We have that $U - x + v$ is also $i$-good. Moreover, $X^+(U-x+v) = X^+(U)$, and $X^-(U-x+v) = X^-(U)$.\qed
    \end{subclaim}

    Now, we prove \cref{colorInBetween}. We give the proof for the case where $W$ is an $(i,r-i-1)$-set of $(X^+(U), X^-(U))$; the second case is similar. Let $U'$ be an arbitrary $(i-1,r-i-1)$-set of $(X^+(U), X^-(U))$ contained in $W$. Note that $U'$ can be obtained from $U$ by replacing some vertices in $h_i^+(U)$ by the same number of vertices in $X^+(U) \sm U$ and then replacing some vertices in $h_i^-(U)$ by the same number of vertices in $X^-(U) \sm U$. Using \cref{goodMoveOne} a finite number of times, we conclude that $U'$ is $i$-good, $X^+(U') = X^+(U)$, and $X^-(U') = X^-(U)$. Then, the claim follows from the definition of $X^+(U')$. \qedhere
    \end{proof}

    \begin{claim} \label{itoi+2i-2}
    We have the following claims.\\
        (1) For every integer $i \in [1,r-3]$, if $S_i \times S_{r-i}$ is used in $\mH$, then $S_{i+2} \times S_{r-i-2}$ is used in $\mH$.\\
        (2) For every integer $i \in [3, r-1]$, if $S_i \times S_{r-i}$ is used in $\mH$, then $S_{i-2} \times S_{r-i+2}$ is used in $\mH$.
    \end{claim}
    \begin{proof}
        For (1), see \cref{fig::itoi+2} for a figure of the proof. Let $e$ be an edge colored with $S_i \times S_{r-i}$. We can fix $U$ to be an $i$-good $(r-2)$-set in $e$. By \cref{colorInBetween}, for every $(i, r-i-1)$-set $W$ of $(X^+(U), X^-(U))$, we have that $W$ is colored with $S_i \times S_{r-i}$, $h^+_i(W) = W \cap X^+(U)$, and $h^-_i(W) = W \cap X^-(U)$. Hence, by the coloring, no $(i+1,r-i-1)$-set can be an edge in $\mH$, so for every $(i+1,r-i-2)$-set $W'$ of $(X^+(U),X^-(U))$, we have
        \[
            d_{X^+(U)}(W') = d(W') \ge \drmh \stackrel{\cref{propertyX}}{>} |X^+(U)| / 2.
        \]
        Therefore, by \cref{manyToColoring}, $S_{i+2} \times S_{r-i-2}$ is used.
        Statement (2) is proved analogously. \qedhere
    \end{proof}
\begin{figure}[ht]
  \centering
  \resizebox{0.48\linewidth}{!}{
    \begin{tikzpicture}[line cap=round, line join=round, font=\small, yscale=0.8]

      \draw[thick] (0,0) rectangle (8,2.2);
      \node[anchor=north east] at (7.9,2.15) {$X^-(U)$};

      \draw[thick] (0,2.8) rectangle (8,5.0);
      \node[anchor=north east] at (7.9,4.95) {$X^+(U)$};

      \coordinate (a1) at (1.1,3.8);
      \coordinate (a2) at (2.7,3.8);
      \coordinate (a3) at (4.3,3.8);
      \coordinate (a4) at (5.9,3.8);
      \node at (a1) {$+$};
      \node at (a2) {$+$};
      \node at (a3) {$+$};
      \node at (a4) {$+$};

      \coordinate (b1) at (1.9,1.05);
      \coordinate (b2) at (2.96,1.05);
      \coordinate (b3) at (4.04,1.05);
      \coordinate (b4) at (5.1,1.05);
      \node at (b1) {$-$};
      \node at (b2) {$-$};
      \node at (b3) {$-$};
      \node at (b4) {$-$};
      
      \draw[thick, rounded corners=6pt] (0.6,4.55) -- (4.95,4.55) -- (5.65,0.35) -- (1.3,0.35) -- cycle;
      \draw[thick, rounded corners=6pt] (2.3,4.2) -- (6.3,4.2) -- (5.3,0.7) -- (1.3,0.7) -- cycle;
      
      \draw[dashed, thick, rounded corners=6pt] (0.35,0.2) rectangle (6.55,4.8);
    \end{tikzpicture}
  }
  \caption{$r = 8$ and $i = 3$. Let $e$ be an edge colored with $S_3 \times S_5$ and $U$ be a $3$-good set in $e$. By \cref{colorInBetween}, we have that all the $(3,4)$-sets of $(X^+(U), X^-(U))$ are of $S_3 \times S_5$, in which vertices in $X^+(U)$ are of $S_3$ and vertices in $X^-(U)$ are of $S_5$. 
  Hence, by the coloring, no $(4,4)$-set of $(X^+(U),X^-(U))$ can be an edge.}
  \label{fig::itoi+2}
\end{figure}
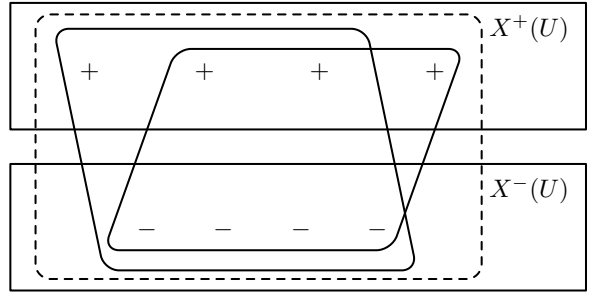
    \begin{claim} \label{nor-2}
        $S_{r-2}\times S_2$ is not used in $\mH$.
    \end{claim}
    \begin{proof}
        Suppose for a contradiction that there is an edge $e \in E(\mH)$ colored with $S_{r-2}\times S_2$. In $e$, we can fix a good $(r-2)$-set $U$. As in the proof of \cref{itoi+2i-2}, we claim that no $(r-1,1)$-set of $(X^+(U), X^-(U))$ can be an edge. Indeed,
        by \cref{colorInBetween}, for every $(r-2, 1)$-set of $(X^+(U), X^-(U))$, we have that $W$ is colored with $S_{r-2} \times S_2$, $h^+_{r-2}(W) = W \cap X^+(U)$, and $h^-_{r-2}(W) = W \cap X^-(U)$. Therefore, by the coloring, no $(r-1,1)$-set can be an edge in $\mH$, as claimed.
        
        Now, for every $(r-1)$-set $W$ of $X^+(U)$, we have $d(W) = d_{X^+(U)}(W)$.
        However, \cref{le1over2} yields $\drmh[\mH[X^+(U)]] \le (1/2+o(1))|X^+(U)|$. Thus, there is an $(r-1)$-set $W$ of $X^+(U)$ such that 
            \[
            d(W) = 
            d_{X^+(U)}(W) \le (1/2+o(1)) \cdot |X^+(U)|
            \stackrel{\cref{propertyX}}{\le} (1/2+o(1)) \cdot (2/3-\eps) n < n/3,
            \]
            a contradiction.
    \end{proof}

    Fix an arbitrary edge $e \in E(\mH)$ and assume that $e$ is colored with $S_{i_0} \times S_{r-i_0}$, so in particular, we have that $S_{i_0} \times S_{r-i_0}$ is used in $\mH$. By \cref{itoi+2i-2,nor-2}, we have that $\{i_0 + 2j : j \in \mathbb{Z}\} \cap [1,r-1] = \{1,3,\ldots, r-1\}$. Therefore, $r$ is even and $\sit$ is available for every odd $i \in [1,r-1]$.
    \qedhere

    \end{proof}

\section{Concluding remarks} \label{sec::conRem}
It is natural to try to push the method of this paper further in order to show that some tight cycles $\C$ have smaller codegree Tur\'an density.
One major difficulty is that, once the minimum codegree drops below $n/3$, it is no longer clear how to obtain a vertex partition as in \cref{propertyX}. 
Moreover, even if one could give an appropriate partition of the vertex set, it seems difficult to derive an analogue of \cref{itoi+2i-2}. 
With only two parts, the absence of certain edges can immediately force the existence of certain types of edges, due to the minimum codegree assumption. However, with three or more parts, then there are many different ways to satisfy the minimum codegree requirement.

As discussed in~\cite[Section 7]{sankar2024turan}, results akin to \cref{coloringLemma} in \cref{sec::ckrfree} hold for the far broader family of ``twisted'' tight cycles, and there are additional coloring results for tight cycles minus edges. We believe that the oriented coloring approach holds a lot of promise for other extremal problems regarding hypergraphs free of these families. Any further results in this direction, particularly determining the (codegree) Tur\'an densities of more cycles in these families would be very interesting.

\paragraph*{Acknowledgments.}
This work was partially done when Luo and Sankar attended the 2025 Joint Mathematics Meetings and when all authors participated in the IAS/Park City Mathematics Institute in Summer 2025.

\setlength{\bibsep}{0pt}


\begin{thebibliography}{10}

\bibitem{balogh2022hypergraph}
J.~Balogh, F.~C. Clemen, and B.~Lidick\'{y}.
\newblock Hypergraph {T}ur\'{a}n problems in {$\ell_2$}-norm.
\newblock In {\em Surveys in combinatorics 2022}, volume 481 of {\em London
  Math. Soc. Lecture Note Ser.}, pages 21--63. Cambridge Univ. Press,
  Cambridge, 2022.

\bibitem{balogh2026positive}
J.~Balogh, A.~Halfpap, B.~Lidick\'{y}, and C.~Palmer.
\newblock Positive co-degree densities and jumps.
\newblock {\em Innov. Graph Theory}, 3:1--36, 2026.

\bibitem{balogh2024turan}
J.~Balogh and H.~Luo.
\newblock Tur\'{a}n density of long tight cycle minus one hyperedge.
\newblock {\em Combinatorica}, 44(5):949--976, 2024.

\bibitem{bodnar2025turan}
L.~Bodn{\'a}r, J.~Le{\'o}n, X.~Liu, and O.~Pikhurko.
\newblock The {T}ur{\'a}n density of short tight cycles.
\newblock {\em arXiv:2506.03223}, 2025.

\bibitem{bucic2023uniform}
M.~Buci\'{c}, J.~W. Cooper, D.~Kr\'{a}\v{l}, S.~Mohr, and D.~Munh\'{a}~Correia.
\newblock Uniform {T}ur\'{a}n density of cycles.
\newblock {\em Trans. Amer. Math. Soc.}, 376(7):4765--4809, 2023.

\bibitem{czygrinow2001codegree}
A.~Czygrinow and B.~Nagle.
\newblock A note on codegree problems for hypergraphs.
\newblock {\em Bull. Inst. Combin. Appl.}, 32:63--69, 2001.

\bibitem{erdos1964extremal}
P.~Erd\H{o}s.
\newblock On extremal problems of graphs and generalized graphs.
\newblock {\em Israel J. Math.}, 2:183--190, 1964.

\bibitem{erdos1966limit}
P.~Erd\H{o}s and M.~Simonovits.
\newblock A limit theorem in graph theory.
\newblock {\em Studia Sci. Math. Hungar.}, 1:51--57, 1966.

\bibitem{erdos1946structure}
P.~Erd\H{o}s and A.~H. Stone.
\newblock On the structure of linear graphs.
\newblock {\em Bull. Amer. Math. Soc.}, 52:1087--1091, 1946.

\bibitem{falgasravry2023codegree}
V.~Falgas-Ravry, O.~Pikhurko, E.~Vaughan, and J.~Volec.
\newblock The codegree threshold of {$K^-_4$}.
\newblock {\em J. Lond. Math. Soc. (2)}, 107(5):1660--1691, 2023.

\bibitem{falgasravry2012turan}
V.~Falgas-Ravry and E.~R. Vaughan.
\newblock Tur\'{a}n {$H$}-densities for 3-graphs.
\newblock {\em Electron. J. Combin.}, 19(3):Paper 40, 26, 2012.

\bibitem{frankl1984exact}
P.~Frankl and Z.~F\"{u}redi.
\newblock An exact result for {$3$}-graphs.
\newblock {\em Discrete Math.}, 50(2-3):323--328, 1984.

\bibitem{han2021covering}
J.~Han, A.~Lo, and N.~Sanhueza-Matamala.
\newblock Covering and tiling hypergraphs with tight cycles.
\newblock {\em Combin. Probab. Comput.}, 30(2):288--329, 2021.

\bibitem{kamcev2024turan}
N.~Kam\v{c}ev, S.~Letzter, and A.~Pokrovskiy.
\newblock The {T}ur\'{a}n density of tight cycles in three-uniform hypergraphs.
\newblock {\em Int. Math. Res. Not. IMRN}, (6):4804--4841, 2024.

\bibitem{keevash2011hypergraph}
P.~Keevash.
\newblock Hypergraph {T}ur\'{a}n problems.
\newblock In {\em Surveys in combinatorics 2011}, volume 392 of {\em London
  Math. Soc. Lecture Note Ser.}, pages 83--139. Cambridge Univ. Press,
  Cambridge, 2011.

\bibitem{keevash2007codegree}
P.~Keevash and Y.~Zhao.
\newblock Codegree problems for projective geometries.
\newblock {\em J. Combin. Theory Ser. B}, 97(6):919--928, 2007.

\bibitem{ma2024codegree}
J.~Ma.
\newblock On codegree {T}ur{\'a}n density of the 3-uniform tight cycle
  {$C_{11}$}.
\newblock {\em arXiv:2409.02765}, 2024.

\bibitem{ma2025codegree}
J.~Ma and M.~Rong.
\newblock The codegree {T}ur{\'a}n density of tight cycles.
\newblock {\em arXiv:2512.23011}, 2025.

\bibitem{mubayi2011hypergraph}
D.~Mubayi, O.~Pikhurko, and B.~Sudakov.
\newblock Hypergraph {T}ur{\'a}n problem: some open questions.
\newblock In {\em AIM workshop problem lists}, page 166, 2011.

\bibitem{mubayi2007codegree}
D.~Mubayi and Y.~Zhao.
\newblock Co-degree density of hypergraphs.
\newblock {\em J. Combin. Theory Ser. A}, 114(6):1118--1132, 2007.

\bibitem{piga2026codegree}
S.~Piga, N.~Sanhueza-Matamala, and M.~Schacht.
\newblock The codegree {T}ur\'{a}n density of 3-uniform tight cycles.
\newblock {\em J. Combin. Theory Ser. B}, 176:1--6, 2026.

\bibitem{razborov2007flag}
A.~A. Razborov.
\newblock Flag algebras.
\newblock {\em J. Symbolic Logic}, 72(4):1239--1282, 2007.

\bibitem{sankar2024turan}
M.~Sankar.
\newblock The {T}ur{\'a}n {D}ensity of 4-{U}niform {T}ight {C}ycles.
\newblock {\em arXiv:2411.01782}, 2024.

\bibitem{sankar2025topological}
M.~Sankar.
\newblock {\em Topological {A}pproaches to {E}xtremal {C}ombinatorics}.
\newblock PhD thesis, Stanford University, 2025.

\bibitem{turan1961research}
P.~Tur\'{a}n.
\newblock Research problems.
\newblock {\em Magyar Tud. Akad. Mat. Kutat\'{o} Int. K\"{o}zl.}, 6:417--423,
  1961.

\end{thebibliography}
\end{document}